\documentclass[a4paper, 10pt, english]{article}
\usepackage{lscape}
\usepackage{pdflscape}

\usepackage{stmaryrd}

\usepackage{cmll}

\usepackage{comment}

\usepackage{amscd}
\usepackage{amssymb,amsmath,amsthm}
\usepackage{mathptmx}
\usepackage{mathrsfs}
\usepackage{color}
\usepackage{xspace}
\usepackage{bussproofs}
\EnableBpAbbreviations

\usepackage{tikz}

\usepackage{txfonts}

\usepackage{centernot}

\usepackage{mathtools}

\usepackage{relsize}

\usepackage{multirow}

\usepackage{multicol}

\usepackage{array}
\newcolumntype{C}[1]{>{\centering\arraybackslash}p{#1}}
\newcolumntype{L}[1]{>{\arraybackslash}p{#1}}

\usepackage{cleveref}

\usetikzlibrary{arrows}
\usetikzlibrary{matrix}
\usetikzlibrary{patterns}
\usetikzlibrary{shapes}
\newcolumntype{C}[1]{>{\centering\arraybackslash}p{#1}}
\newcolumntype{L}[1]{>{\arraybackslash}p{#1}}
\DeclareMathSymbol{\Gamma}{\mathalpha}{operators}{0}
\DeclareMathSymbol{\Delta}{\mathalpha}{operators}{1}
\DeclareMathSymbol{\Theta}{\mathalpha}{operators}{2}
\DeclareMathSymbol{\Lambda}{\mathalpha}{operators}{3}
\DeclareMathSymbol{\Xi}{\mathalpha}{operators}{4}
\DeclareMathSymbol{\Pi}{\mathalpha}{operators}{5}
\DeclareMathSymbol{\Sigma}{\mathalpha}{operators}{6}
\DeclareMathSymbol{\Upsilon}{\mathalpha}{operators}{7}
\DeclareMathSymbol{\Phi}{\mathalpha}{operators}{8}
\DeclareMathSymbol{\Psi}{\mathalpha}{operators}{9}
\DeclareMathSymbol{\Omega}{\mathalpha}{operators}{10}
\newtheorem{theorem}{Theorem}
\newtheorem{proposition}{Proposition}
\newtheorem{lemma}{Lemma}

\newtheorem{definition}{Definition}

\def\mc{\multicolumn}

\def\fCenter{{\mbox{$\ \vdash\ $}}}

\newcommand{\fns}{\footnotesize}

\def\mc{\multicolumn}

\newcommand{\xtop}{\ensuremath{\top}\xspace}
\newcommand{\xbot}{\ensuremath{\bot}\xspace}
\newcommand{\xand}{\ensuremath{\wedge}\xspace}
\newcommand{\xor}{\ensuremath{\vee}\xspace}
\newcommand{\xneg}{\ensuremath{\neg}\xspace}
\newcommand{\xrarr}{\ensuremath{\rightarrow}\xspace}
\newcommand{\xcrarr}{\ensuremath{\,{>\mkern-7mu\raisebox{-0.065ex}{\rule[0.5865ex]{1.38ex}{0.1ex}}}\,}\xspace}

\newcommand{\XTOP}{\ensuremath{\hat{\top}}\xspace}
\newcommand{\XBOT}{\ensuremath{\check{\bot}}\xspace}
\newcommand{\XAND}{\ensuremath{\:\hat{\wedge}\:}\xspace}
\newcommand{\XOR}{\ensuremath{\:\check{\vee}\:}\xspace}
\newcommand{\XNEG}{\ensuremath{\:\tilde{\neg}\:}\xspace}
\newcommand{\XRARR}{\ensuremath{\:\check{\rightarrow}\:}\xspace}
\newcommand{\XCRARR}{\ensuremath{\hat{{\:{>\mkern-7mu\raisebox{-0.065ex}{\rule[0.5865ex]{1.38ex}{0.1ex}}}\:}}}\xspace}

\newcommand{\Ineg}{\ensuremath{{\sim}}\,}

\newcommand{\Itop}{\ensuremath{1_\mathrm{I}}\xspace}
\newcommand{\Ibot}{\ensuremath{0_\mathrm{I}}\xspace}

\newcommand{\Iand}{\ensuremath{\cap}\xspace}
\newcommand{\Ior}{\ensuremath{\cup}\xspace}
\newcommand{\Irarr}{\ensuremath{\,{\raisebox{-0.065ex}{\rule[0.5865ex]{1.38ex}{0.1ex}}\mkern-5mu\supset}\,}\xspace}
\newcommand{\Icrarr}{\ensuremath{\,{\supset\mkern-5.5mu\raisebox{-0.065ex}{\rule[0.5865ex]{1.38ex}{0.1ex}}}\,}\xspace}

\newcommand{\INEG}{\,\ensuremath{\sim}\,\,}

\newcommand{\ITOP}{\ensuremath{\hat{1}_\mathrm{I}}\xspace}
\newcommand{\IBOT}{\ensuremath{\check{0}_\mathrm{I}}\xspace}
\newcommand{\IAND}{\ensuremath{\:\hat{\cap}\:}\xspace}
\newcommand{\IOR}{\ensuremath{\:\check{\cup}\:}\xspace}
\newcommand{\IRARR}{\ensuremath{\check{\,{\raisebox{-0.065ex}{\rule[0.5865ex]{1.38ex}{0.1ex}}\mkern-5mu\supset}\,}}\xspace}
\newcommand{\ICRARR}{\ensuremath{\hat{\,{\supset\mkern-5.5mu\raisebox{-0.065ex}{\rule[0.5865ex]{1.38ex}{0.1ex}}}\,}}\xspace}

\newcommand{\Cneg}{\,\ensuremath{-}\,\,}
\newcommand{\Ctop}{\ensuremath{1}_\mathrm{C}\,\xspace}
\newcommand{\Cbot}{\ensuremath{0}_\mathrm{C}\,\xspace}
\newcommand{\Cand}{\ensuremath{\sqcap}\xspace}
\newcommand{\Cor}{\ensuremath{\sqcup}\xspace}
\newcommand{\Crarr}{\ensuremath{\,{\raisebox{-0.065ex}{\rule[0.5865ex]{1.38ex}{0.1ex}}\mkern-5mu\sqsupset}\,}\xspace}
\newcommand{\Ccrarr}{\ensuremath{\,{\sqsupset\mkern-5.5mu\raisebox{-0.065ex}{\rule[0.5865ex]{1.38ex}{0.1ex}}}\,}\xspace}

\newcommand{\CNEG}{\,\ensuremath{\:\tilde{-}}\,\,}

\newcommand{\CTOP}{\ensuremath{\hat{1}_\mathrm{C}}\xspace}
\newcommand{\CBOT}{\ensuremath{\check{0}_\mathrm{C}}\xspace}
\newcommand{\CAND}{\ensuremath{\:\hat{\sqcap}\:}\xspace}
\newcommand{\COR}{\ensuremath{\:\check{\sqcup}\:}\xspace}
\newcommand{\CRARR}{\ensuremath{\check{\,{\raisebox{-0.065ex}{\rule[0.5865ex]{1.38ex}{0.1ex}}\mkern-5mu\sqsupset}\,}}\xspace}
\newcommand{\CCRARR}{\ensuremath{\hat{\,{\sqsupset\mkern-5.5mu\raisebox{-0.065ex}{\rule[0.5865ex]{1.38ex}{0.1ex}}}\,}}\xspace}

\newcommand{\wbox}{\ensuremath{\circ}_C\xspace}
\newcommand{\wdia}{\ensuremath{\circ}_I\xspace}

\newcommand{\bboxr}{\ensuremath{\blacksquare}_I\,\xspace}
\newcommand{\bboxl}{\ensuremath{\Diamondblack}_I\,\xspace}
\newcommand{\bdiar}{\ensuremath{\blacksquare}_C\,\xspace}
\newcommand{\bdial}{\ensuremath{\Diamondblack}_C\,\xspace}
\newcommand{\WBOX}{\ensuremath{\tilde{\circ}_C\,}\xspace}

\newcommand{\WDIA}{\ensuremath{\:\tilde{\circ}_I\,}\xspace}
\newcommand{\BBOXR}{\ensuremath{\:\check{\blacksquare}_I\,}\xspace}
\newcommand{\BBOXL}{\ensuremath{\:\hat{\Diamondblack}_I\,}\xspace}
\newcommand{\BDIAL}{\ensuremath{\:\hat{\Diamondblack}_C\,}\xspace}
\newcommand{\BDIAR}{\ensuremath{\:\check{\blacksquare}_C\,}\xspace}

\usepackage{authblk}
\EnableBpAbbreviations
\def\fCenter{{\mbox{$\ \vdash\ $}}}

\makeindex

\title{Proper Multi-Type Display Calculi for Rough Algebras}

 \author[1]{Giuseppe Greco}
   \author[2]{Fei Liang}
   \author[3]{Krishna Manoorkar}
   \author[2,4]{Alessandra Palmigiano\thanks{The research of the second and fourth author is supported by the NWO Vidi grant 016.138.314, the NWO Aspasia grant 015.008.054, and a Delft Technology Fellowship awarded to the second author in 2013.}}
  \affil[1]{Utrecht University, the Netherlands}
  \affil[2]{Delft University of Technology, the Netherlands}
  \affil[3]{Indian Institute of Technology, India}
  \affil[4]{University of Johannesburg, South Africa}
    \date{}
\begin{document}
\maketitle

    \begin {abstract}
    In the present paper, we endow the logics of  topological quasi Boolean algebras, topological quasi Boolean algebras 5, intermediate algebras of types 1-3, and pre-rough algebras with proper multi-type display calculi which are sound, complete, conservative, and enjoy cut
elimination and subformula property. Our proposal builds on an algebraic analysis and  applies the principles of the multi-type methodology
in the design of display calculi.
  \end{abstract}

\section{Introduction}\label{sec:Introduction}
Rough algebras and related structures arise in tight connection with formal models of imperfect information \cite{pawlak1982rough}, and have been investigated for more than twenty years using techniques from universal algebra and algebraic logic, giving rise to a rich theory (cf.~e.g.~\cite{banerjee1996rough,iwinski1987algebraic,comer1995perfect,saha2014algebraic,saha2016algebraic}). Sound and complete sequent calculi have been introduced for the logics naturally associated with some of these classes of algebras \cite{saha2014algebraic,saha2016algebraic}. However, the cut rule in these calculi is not eliminable. Very recently, sequent calculi with cut elimination and a non-standard version of subformula property have been introduced in \cite{MinChaZhe18} for some of these logics, but not for the logic of the so-called {\em intermediate algebras of type 3} (cf.~\cite{saha2014algebraic}, Definition 2.11). In these calculi, the subformula property is non-standard because each logical connective has four introduction rules, two of which are non-standard and introduce the given logical connective under the scope of negation. 

In the present paper, we introduce a family of proper display calculi for the logics associated with the classes of `{\em rough algebras}'\footnote{Although the name `rough algebras' has a specific meaning  in this literature (reported in Definition \ref{def:dm}), in the present paper we find it convenient to use it as the generic name for the class of topological quasi Boolean algebras and its subclasses.} discussed in \cite{saha2014algebraic}; namely, topological quasi Boolean algebras (tqBa), topological quasi Boolean algebras 5 (tqBa5), intermediate algebras of types 1-3 (IA1, IA2, IA3), and pre-rough algebras (pra), cf.~Definition \ref{def:dm}. 

Our methodology is very akin to the spirit of \cite{banerjee1996rough}, is driven by algebraic considerations, and is grounded on the general results and insights of the theory of {\em multi-type calculi}, introduced in \cite{Multitype,PDL,TrendsXIII} and motivated by \cite{GKPLori,ProofTheoreticDEL}. This theory has proven effective in endowing many well known but proof-theoretically challenging logical systems (cf.~e.g.~\cite{linearlogPdisplayed,greco2017multi,GrecoPalmigianoLatticeLogic,bilattice,fei2018,apostolos2018,Inquisitive}) with sequent calculi enjoying the  excellent properties mentioned in the abstract, which hold in full uniformity and are guaranteed by the general theory. This theory modularly covers also a wide class of axiomatic extensions of given logics \cite{greco2016unified}, and therefore has provided a powerful and flexible algebraic and proof-theoretic environment for the design of new families of logics of agency and coordination (cf.~\cite{BGPTW}), which introduces novel applications of  non-classical logics  to formalization problems in different fields, such as the social sciences. 

The first contribution of the present paper is an equivalent presentation of rough algebras, based on so-called {\em heterogeneous algebras} \cite{birkhoff1970heterogeneous}. Intuitively, heterogeneous algebras are algebras with more than one domain, and their operations might span across different domains. The classes of heterogeneous algebras corresponding to rough algebras have three domains, respectively corresponding to (abstract representations of) {\em general} sets and upper and lower  {\em definable} sets of an approximation space. Each of these three domains corresponds to a  distinct  {\em type}.  The modal operators capturing the lower and upper definable approximations of a general set are then modeled as heterogeneous maps from the general type to one of  the two definable types. The  equivalent heterogeneous presentations of rough algebras come naturally equipped with a multi-type logical language, and are characterized by   axiomatizations which can be readily recognized to be {\em analytic inductive} (cf.~\cite[Definition 55]{greco2016unified}), and hence, by the general theory of multi-type calculi,  can be effectively captured by  {\em proper} multi-type display calculi which are sound, complete, conservative, and enjoy cut elimination and  {\em standard} subformula property, given that the introduction rules for all connectives are standard. The introduction of these calculi is the second contribution of the present paper. 

Compared with \cite{MinChaZhe18}, the multi-type methodology allows for more modularity, which not only has made it possible to account for the logic of IA3, but will also make it possible to extend the present theory so as to cover  weaker versions of rough algebras based on e.g.~semi De Morgan algebras \cite{greco2017multi}, or even general lattices \cite{KriPalNac18}, which will account for the proof-theoretic aspects of the logics of rough concepts.
\section{Preliminaries}

\subsection{Varieties of rough algebras}\label{ssec: varieties}

\begin{definition}(cf.~Section 2 \cite{saha2016algebraic})\label{def:dm}
$\mathbb{T} = (\mathbb{L},  I)$ is a {\em topological quasi-Boolean algebra} (tqBa) if $\mathbb{L} = (\mathrm{L}, \lor,\land,\neg,  \top, \bot)$ is a De Morgan algebra and for all $a, b \in \mathrm{L}$,
 \vspace{-0.3em}
\begin{center}
\begin{tabular}{lll l}
 T1. $I(a \land b)= Ia \land Ib$, &T2. $IIa = Ia$, & T3. $Ia \leq a$, &T4. $I \top = \top$. \\
\end{tabular}
\end{center}
For any $a\in \mathbb{T}$, let  $Ca: =  \neg I\neg a$. We consider the subclasses of tqBas defined as in the following table.
\begin{center}
\begin{tabular}{|c|c|c|}
\hline
Algebras & Acronyms & Axioms\\
\hline
{\em topological quasi Boolean algebra 5} & tqBa5 & T5: $CIa = Ia$\\
 \hline
{\em intermediate algebra of type 1} & IA1 & T5, T6: $Ia \lor \neg Ia = \top$\\
 \hline
 {\em intermediate algebra of type 2} &IA2 & T5, T7: $Ia \lor Ib = I(a \lor b)$\\
 \hline
 {\em intermediate algebra of type 3}  &IA3 &T5, T8: $Ia \leq Ib$  and $Ca \leq Cb$ imply $a \leq b$\\
 \hline
 {\em pre-rough algebra}  &pra & T5, T6, T7, T8. \\
 \hline
 \end{tabular}
 \end{center}
 
A  {\em rough algebra} is a complete and completely distributive pre-rough algebra.

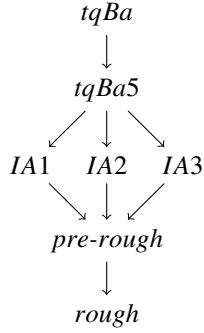
\begin{center}
\begin{tikzpicture}[node/.style={circle, draw, fill=black}, scale=1]
\node (IA2) at (0,-0.5) {${IA2}$};
\node (IA1) at (-1,-0.5) {${IA1}$};
\node (IA3) at (1,-0.5) {${IA3}$};
\node (5) at (0,0.5) {${tqBa5}$};
\node (tq) at (0,1.5) {${tqBa}$};
\node (pre) at (0,-1.5) {${pre}$-${rough}$};
\node (rough) at (0,-2.5) {${rough}$};
\draw [->] (tq)  to (5);
\draw [->] (5) to  (IA1);
\draw [->] (5) to  (IA2);
\draw [->] (5) to  (IA3);
\draw [->] (IA1) to  (pre);
\draw [->] (IA2) to  (pre);
\draw [->] (IA3) to  (pre);
\draw [->] (pre) to  (rough);
\end{tikzpicture}
\end{center}
\end{definition}

\begin{lemma}
\label{cor:interior-closure}
Any tqBa $\mathbb{T} = (\mathrm{L}, \lor,\land,\neg, I, \top, \bot)$ satisfies the following equalities:
\begin{center}
\begin{tabular}{lll}
(i)\ $I(Ia \lor Ib) = Ia \lor Ib$ &$\quad\quad$ & (ii) $C(Ca \land Cb) = Ca \land Cb$.\\
\end{tabular}
\end{center}
\end{lemma}
\begin{proof}
(i) $I(Ia \lor Ib) \leq Ia \lor Ib$ is a straightforward consequence of T3. As to the converse direction, it is enough to show that  $Ia \leq I(Ia \lor Ib)$ and $Ib \leq I(Ia \lor Ib)$. Let us show the first of these inequalities. From T1 it immediately follows that $I$ is monotone. Hence $Ia \leq Ia \lor Ib$ implies $IIa \leq I(Ia \lor Ib)$. Hence, by T2, $Ia \leq IIa\leq I(Ia \lor Ib)$. Analogously one proves $Ib \leq I(Ia \lor Ib)$. The proof for (ii) is dual.
\end{proof}
In what follows, we use the abbreviated names of the algebras written in ``blackboard bold''  (e.g.~$\mathbb{TQBA}$, etc.) to indicate their corresponding classes. When it is unambiguous, we will use {\em rough algebras} as the generic name for these classes.

\subsection{The logics of rough algebras}\label{ssec:logics}
Fix a denumerable set $\mathsf{Atprop}$ of propositional variables, let $p$ denote an element in $\mathsf{Atprop}$. The logics of rough algebras share the language $\mathcal{L}$ which is defined recursively as follows:
$$A ::= p \mid \top \mid \bot \mid \neg A \mid IA \mid CA \mid A \land A \mid A \lor A.$$

\begin{definition}
The logic $\mathrm{H.TQBA}$ of the class  $\mathbb{TQBA}$ is defined by adding the following axioms to  De Morgan logic:
$$IA \fCenter A, \quad IA \fCenter IIA, \quad I(A \wedge B) \fCenter IA \wedge IB, \quad IA \wedge IB \fCenter I(A \wedge B), \quad \top \fCenter I\top$$
$$CA \fCenter \neg I \neg A, \quad \neg I \neg A \fCenter CA, \quad  \neg C \neg A \fCenter IA, \quad IA \fCenter \neg C \neg A.$$ We consider the following extensions of $\mathrm{H.TQBA}$ corresponding to the subclasses of $\mathbb{TQBA}$ reported above:

\begin{center}
\begin{tabular}{|c|c|p{5cm}|}
\hline
Class of algebras &name of logic &  Axioms/Rules\\
\hline
$\mathbb{TQBA}5$& $\mathrm{H.TQBA5}$ &1: $CIA \fCenter IA$\\
 \hline
$\mathbb{IA}1$&$\mathrm{H.IA1}$&1, 2: $ \top \fCenter IA \vee \neg IA$\\
 \hline
$\mathbb{IA}2$&$\mathrm{H.IA2}$&1, 3: $I(A \vee B) \fCenter IA$ \\
 \hline
 $\mathbb{IA}3$&$\mathrm{H.IA3}$&
 \rule{0pt}{4ex}\noindent 1, 4:
\AX $IA \fCenter IB $
\AX $CA \fCenter CB$
\BI  $A \fCenter B $
\DP

\\
 \hline
 $\mathbb{PRA}$ &$\mathrm{H.PRA}$ &1, 2, 3, 4\\
 \hline
 \end{tabular}
 \end{center}
\end{definition}
Let $\mathrm{H}$ denote any of the logics in the table above (second column), and $\mathbb{A}$ denote its corresponding class of algebras in the table above (first column, same row as $\mathrm{H}$). 
\begin{theorem}[Completeness]\label{completeness:H.SDM}
$\mathrm{H}$ is sound and complete with respect to $\mathbb{A}$, that is, if  $A \vdash B$ is an $\mathcal{L}$-sequent, then $A\vdash B$ is  derivable in $\mathrm{H}$ iff $h(A)\leq h(B)$ for all $\mathbb{T}\in \mathbb{A}$ and every interpretation $h:\mathcal{L}\to \mathbb{T}$.
\end{theorem}

\section{Towards a multi-type presentation: algebraic analysis}
In this section, we   equivalently represent rough algebras as heterogeneous algebras.
\subsection{The kernels of algebras}\label{ssec:kernel}
For any tqBa $\mathbb{T}$ (cf.~Definition \ref{def:dm}),  we let $\mathrm{K}_I:= \{ Ia \mid a \in \mathrm{L} \}$ and $\mathrm{K}_C:=\{ Ca  \mid a \in \mathrm{L}\}$, and let $\iota: \mathrm{L} \rightarrow \mathrm{K}_I$ and $\gamma: \mathrm{L} \rightarrow  \mathrm{K}_C$ be defined by the assignments  $a\mapsto Ia$ and $a \mapsto Ca$, respectively. Let $e_I:  \mathrm{K}_I \hookrightarrow \mathrm{L}$ and $e_C:  \mathrm{K}_C \hookrightarrow \mathrm{L}$ denote the natural embeddings.
Axioms T1, T2, and T3 imply that $I: \mathrm{L} \rightarrow \mathrm{L}$ is an interior operator and $C:  \mathrm{L} \rightarrow \mathrm{L}$ is a closure operator on $\mathrm{L}$ seen as a poset. Hence, by general order-theoretic facts (cf.~\cite[Chapter 7]{davey2002introduction}),  $e_I$ (resp.~$e_C$) is the left (resp.~right) adjoint of $\iota$ (resp.~$\gamma$), in symbols: $e_I \dashv \iota$ and $\gamma \dashv e_C$, i.e.~for any $\alpha\in \mathrm{K}_I$, $\xi\in \mathrm{K}_C$ and $a\in \mathrm{L}$,

\begin{equation}\label{eq:adjuction}
e_I(\alpha)\leq a \quad\mbox{ iff }\quad \alpha\leq \iota(a)\quad\quad\quad\quad \gamma(\xi)\leq a \quad\mbox{ iff }\quad \xi\leq e_C(a).
\end{equation}
The following equations are  straightforward consequences of the definitions of the maps and \eqref{eq:adjuction}:
\begin{equation}\label{eq:Interior}
 \iota(e_I(\alpha)) = \alpha \quad\quad   e_I(\iota(a)) = Ia\quad\quad\quad\quad
\gamma(e_C(\xi)) = \xi\quad \quad e_C(\gamma(a)) = Ca.
\end{equation}
The following lemma is a straightforward consequence of the definitions.

\begin{definition}
\label{def:kernel}
For any tqBa $\mathbb{T} = (\mathrm{L}, \lor,\land, I, \neg, \top, \bot)$, the {\em left-kernel} $\mathbb{K}_{\mathbb{I}}= (\mathrm{K}_I, \cup, \cap,  1_\mathbb{I}, 0_\mathbb{I})$  and the {\em right-kernel} $\mathbb{K}_{\mathbb{C}} = (\mathrm{K}_C, \sqcup, \sqcap,  1_\mathbb{C}, 0_\mathbb{C})$ are such that, for all $\alpha, \beta\in K_I$, and all $\xi, \chi\in K_C$,
\begin{center}
\begin{tabular}{llcll}
K1. &  $\alpha \cup \beta: = \iota(e_I (\alpha) \lor e_I(\beta))$ &$\quad\quad$& K$'$1. & $\xi \sqcup \chi: = \gamma(e_C (\xi) \lor e_C(\chi))$\\
K2.& $\alpha \cap \beta: = \iota(e_I (\alpha) \land e_I(\beta))$ && K$'$2. &  $\xi \sqcap \chi: = \gamma(e_C(\xi) \land e_C(\chi))$\\
K3.&  $1_I:  = \iota(\top)$; && K$'$3. & $1_C:  = \gamma(\top)$\\
K4.&  $0_I: = \iota(\bot)$ && K$'$4.  & $0_C = \gamma(\bot)$.\\
\end{tabular}
\end{center}
If $\mathbb{T}$ is a tqBa5, we define $\sim: K_I\rightarrow K_I$  and $-: K_C\rightarrow K_C$ by the following equation:
\begin{center}
\begin{tabular}{llcll}
K5. & $\sim\alpha := \iota\neg e_I(\alpha)$ && K$'$5. & $- \xi := \gamma \neg e_C \xi$ \\
\end{tabular}
\end{center}
\end{definition}



\begin{lemma}\label{prop: algebraic structure on kernels}
For any  tqBa $\mathbb{T}$,
\begin{enumerate}
\item[1.] $\iota : \mathbb{T} \twoheadrightarrow \mathbb{K}_I $ and $ \gamma : \mathbb{T}\twoheadrightarrow\mathbb{K}_C$  are surjective maps which satisfy the following equations: for all $a, b \in \mathrm{L}$,

\begin{enumerate}
\item $\iota (a) \cap  \iota(b) =  \iota(a \land b)$, \ \ \ \ $\iota(\top) = 1_I$, \ \ \ \  $\iota(\bot) = 0_I$;
 \item $\gamma(a) \cup  \gamma(b) =  \gamma(a \lor b)$, \ \ \ \ $\gamma(\top) = 1_C$, \ \ \ \  $\gamma(\bot) = 0_C$.
\end{enumerate}
\end{enumerate}
\begin{enumerate}
\item[2.] $e_I:  \mathbb{K}_I \rightarrow \mathbb{T}$ and $e_C:  \mathbb{K}_C \rightarrow \mathbb{T}$  are injective maps which satisfy the following equations: for all $\alpha, \beta\in K_I$, and all $\xi, \chi\in K_C$,
\begin{enumerate}
\item $ e_I(\alpha) \land  e_I(\beta) =  e_I (\alpha \cap \beta)$,\ \ \ \ $ e_I(\alpha) \lor  e_I(\beta) =  e_I (\alpha \cup \beta)$;
\item $ e_C(\xi) \land  e_C(\chi) =  e_C (\xi \sqcap \chi)$,\ \ \ \ $ e_C(\xi) \lor  e_C(\chi) =  e_C(\xi \sqcup \chi)$;
\item  $e_I(1_I) = \top$, \ \ \ \ $e_I(0_I) = \bot$, \ \ \ \ $e_C(1_C) = \top$, \ \ \ \ $e_C(0_C) = \bot$.
\end{enumerate}
\end{enumerate}
\end{lemma}
\begin{proof}
We only prove 1(a) and 2(a), the arguments for 1(b) and 2(b) being dual. The identities in 2(c) easily follow using K3, K4, K$'$3, K$'$4 and the definition of $\mathbb{T}$. The surjectivity of $\iota$ is an immediate consequence of the definition of $K_I$ (cf.~beginning of Section \ref{ssec:kernel}). In what follows, we show that $\iota$ satisfies 1(a).
{\fns
\begin{center}
\begin{tabular}{llll}
$\iota(a) \cap \iota(b)$&$=$&$ \iota (e\iota(a) \land  e\iota(b))$ &  K2\\
&$=$&$\iota(Ia \land Ib)$& \eqref{eq:Interior} \\
&$=$&$ \iota I(a \land b)$ &  T1\\
&$=$&$\iota e\iota(a \land b)$ & \eqref{eq:Interior} \\
&$=$&$\iota(a \land b)$ & \eqref{eq:Interior} \\
\end{tabular}
\end{center}
}

The remaining identities in 1(a) can be shown analogously using K3 and K4. Let us show that  $e_I$ satisfies 2(a) and 2(c). For any $\alpha , \beta \in K_I$, let $a,b \in\mathrm{L}$ be such that $\alpha = \iota(a)$ and $\beta = \iota(b)$.
{\fns
\begin{center}
\begin{tabular}{llllcllll}
$e(\alpha \cap \beta)$&=&$e(\iota(a) \cap \iota(b))$ &  ($\alpha = \iota(a)$, $\beta = \iota(b)$) &&$e(\alpha \cup \beta)$&=&$e(\iota(a) \cup \iota(b))$& ($\alpha = \iota(a)$, $\beta = \iota(b)$) \\
&$=$& $e\iota (e\iota(a) \land e\iota(b))$ &  K2&&&$=$&$ e\iota(e\iota(a) \lor e(\iota(b))))$& by K1\\
&$=$& $I(Ia \land Ib)$ &  \eqref{eq:Interior}&&&$=$&$I(Ia \lor Ib)$ & \eqref{eq:Interior} \\
&$=$& $IIa \land IIb$ &  T1&&&$=$&$ Ia \lor Ib$ &  Lemma \ref{cor:interior-closure}\\
&$=$& $Ia \land Ib$ &  T2&&&$=$&$ e\iota(a) \lor e\iota(b)$ & \eqref{eq:Interior} \\
&$=$& $e\iota(a) \land e\iota(b)$ & \eqref{eq:Interior}&&&$=$&$ e(\alpha) \lor e(\beta)$ & ($\alpha = \iota(a)$, $\beta = \iota(b)$) \\
&$=$& $e(\alpha) \land e(\beta)$ & ($\alpha = \iota(a)$, $\beta = \iota(b)$) && \\
\end{tabular}
\end{center}
}
\end{proof}
\begin{proposition}\label{def:equivalence of kernels}
If $\mathbb{T}$ is a tqBa5, then $\mathbb{K}_I \cong \mathbb{K}_C$.
\end{proposition}
\begin{proof}
Let $f: K_I \to K_C$ be defined as $f:= \gamma e_I$. To show that $f$ is surjective, let $\xi \in K_C$, and let $\xi = \gamma a$ for some $a \in L$, 
{\fns
\begin{center}
\begin{tabular}{rcll}
$\gamma a$ & $=$ & $\gamma e_C \gamma a$ &  \eqref{eq:Interior}\\
                     & $=$ & $\gamma C a$ & \eqref{eq:Interior} \\
                     & $=$ & $\gamma IC a$ & dual of T5\\
                     & $=$ & $\gamma e_I \iota e_C \gamma a$ & \eqref{eq:Interior}\\
                     & $=$ & $\gamma e_I \iota e_C \xi$ & ($\xi = \gamma a$)\\
                     & $=$ & $f(\iota e_C \xi)$ & ($f:= \gamma e_I$)\\
\end{tabular}
\end{center}
}
Since both $\gamma $ and $e_I$ are monotone, so is $f: = \gamma e_I$. To finish the proof, we need to show that for all $\alpha, \beta\in K_I$,  if 
$\gamma e_I(\alpha) \leq \gamma e_I(\beta)$, then $\alpha \leq \beta$. Since $e_C$ is an order embedding, the assumption can be equivalently rewritten as $e_C\gamma e_I(\alpha) \leq e_C\gamma e_I(\beta)$, Let $a, b\in L$ such that $\alpha = \iota a$  and $\beta = \iota b$. Then we can equivalently rewrite the assumption as $e_C\gamma e_I\iota a \leq e_C\gamma e_I\iota b$.  Since $I: = e_I\iota$ and $C: = e_C\gamma$, we can again equivalently rewrite the assumption as $C I a \leq C I b$, and hence, by T5, as $I a \leq Ib$, that is, $e_I \iota a \leq e_I \iota b$. Since $e_I$ is an order-embedding, this yields  $\iota a \leq \iota b$, that is, $\alpha \leq \beta$, as required. This finishes the proof that $\mathbb{K}_I \cong \mathbb{K}_C$ as lattices. Finally, we need to show that $f(\sim \alpha) = - f(\alpha)$ for any $\alpha\in K_I$. For such an  $\alpha$, let $a\in L$ s.t.~$\alpha = \iota(a)$.
{\fns
\begin{center}
\begin{tabular}{cc}
\begin{tabular}{rcll}
$f(\sim \alpha)$ & $=$ & $\gamma e_I \sim \alpha$ & \\
                          & $=$ & $\gamma e_I \sim \iota(a)$ & \\
                          & $=$ & $\gamma e_I \iota \neg e_I \iota(a)$ & \\
                          & $=$ & $\gamma e_I \iota \neg I (a)$ & \\
                          & $=$ & $\gamma I \neg I (a)$ & \\
                          & $=$ & $\gamma I C \neg  (a)$ & \\
                          & $=$ & $\gamma C \neg  (a)$ & \\
\end{tabular}
 & 
\begin{tabular}{rcll}
$- f(\alpha)$ & $=$ & $\gamma \neg e_C f (\alpha)$ & \\
              & $=$ & $\gamma \neg e_C \gamma e_I (\alpha)$ & \\
              & $=$ & $\gamma \neg e_C \gamma e_I \iota (a)$ & \\
              & $=$ & $\gamma \neg C e_I \iota (a)$ & \\
              & $=$ & $\gamma \neg C I (a)$ & \\
              & $=$ & $\gamma I \neg I (a)$ & \\
              & $=$ & $\gamma I C \neg (a)$ & \\
              & $=$ & $\gamma C \neg (a)$ & \\
\end{tabular}
 \\
\end{tabular}
\end{center}
}
\end{proof}

\noindent By the proposition above, we can drop the subscripts in $\mathbb{K}_I$ (or $\mathbb{K}_C$)  and in  $e_I$ and $ e_C$, and refer to $\mathbb {K}$ as the {\em kernel} of  $\mathbb {T}$. The following lemma are  straightforward consequences of K5:
\begin{lemma}
(1) If $\mathbb{T}$ is a tqBa5, then $e({\sim}\alpha) = \neg e(\alpha)$;\\
(2) If $\mathbb{T}$ is an IA1, then $\iota(a \lor b) = \iota(a) \cup \iota(b)$.
\end{lemma}

\begin{proposition}\label{prop:dmba}
If $\mathbb{T}$ is a  tqBa5, then $\mathbb{K}$ is a De Morgan algebra. Moreover, if $\mathbb{T}$ is an IA1, then $\mathbb{K}$ is a Boolean algebra.
\end{proposition}
\begin{proof}
For any $\alpha, \beta \in K_I$, let $a, b \in L$ such that $\alpha = \iota(a)$ and $\beta = \iota(b)$. Let us show that $\sim\sim\alpha = \alpha$ and $\sim(\alpha \cup \beta) = \sim\alpha\, \cap \sim\beta$.
{\fns
\begin{center}
\begin{tabular}{llllcllll}
$\sim\sim\alpha$&=&$\iota\neg e(\iota\neg e(\alpha))$&  K5&&$\sim(\alpha \cup \beta)$ &$=$&$\iota\neg e(\alpha \cup \beta)$ & K5 \\
&=&$\iota\neg e(\iota\neg e(\iota(a))$& (i) &&&=&$\iota\neg e(\iota(e(\alpha) \lor e(\beta)))$ & K1 \\
&$= $&$\iota\neg I \neg Ia$&  \eqref{eq:Interior}&&&=&$\iota\neg e(\iota(e(\iota(a) \lor e(\iota(b))))$ & ($\alpha = \iota(a)$, $\beta = \iota(b)$) \\
&$= $&$\iota CIa$& $C = \neg I \neg$ &&&=&$\iota\neg I(Ia \lor Ib)$ & \eqref{eq:Interior} \\
&$= $&$\iota Ia$&T5 && &=&$\iota\neg(Ia \lor Ib)$ & Lemma \ref{cor:interior-closure} \\
&$=$&$\iota e(\iota(a))$ &  \eqref{eq:Interior}&&&=&$\iota(\neg Ia \land \neg Ib)$& definition of $\mathbb{L}$ \\
&$=$&$ \iota(a)$ & \eqref{eq:Interior} &&&=&$\iota(C\neg a \land C\neg b)$&$C = \neg I \neg$ and definition of $\mathbb{L}$ \\
&$=$&$ \alpha$ & (i) &&&=&$\iota(IC\neg a \land IC\neg b)$&dual of T5 \\
&&&&&&=&$\iota(I\neg Ia \land I\neg Ib)$& $C = \neg I \neg$ and definition of $\mathbb{L}$ \\
&&&&&&=&$\iota(e\iota\neg e\iota(a) \land e\iota\neg e\iota(b))$&\eqref{eq:Interior} \\
&&&&&&=&$\iota(e\iota\neg e(\alpha) \land e\iota\neg e(\beta))$& ($\alpha = \iota(a)$, $\beta = \iota(b)$) \\
&&&&&&=&$\iota(e(\sim\alpha) \land e(\sim\beta))$& K5\\
&&&&&&=&$\sim\alpha\, \cap \sim\beta$& K2\\

\end{tabular}
\end{center}
}

\noindent Using K5, K3, \eqref{eq:Interior} and T7, one can show the identities $\sim 1_I = \iota\neg e(1_I) =   \iota\neg e(\iota(\top)) = \iota\neg I \top = \iota\bot = 0_I$. The argument for  $\sim 0_I =  1_I$ can be given dually. Hence, $\mathbb{K}_I$ is a De Morgan algebra. If $\mathbb{T}$ is an IA1, in order to show that $\mathbb{K}_I$ is a Boolean algebra, we only need to show $\sim\alpha \cup \alpha = 1_I$.
{\fns
\begin{center}
\begin{tabular}{llll}
  $\sim\alpha \cup \alpha$ & =& $\iota\neg e(\alpha) \cup  \alpha$ & K5\\
  &$=$&$\iota(e\iota\neg e(\alpha) \lor  e(\alpha))$ &  K1\\
  &$=$&$\iota(e\iota\neg e\iota(a) \lor  e\iota(a))$ & (i)\\
 &$=$&$\iota(I\neg Ia \lor  Ia)$ & \eqref{eq:Interior} \\
 &$=$&$\iota(\neg CIa \lor  Ia)$ &$C = \neg I \neg$ \\
 &$=$&$\iota(\neg Ia \lor  Ia)$ &T5 \\
 &$=$&$\iota(\top)$ &T6 \\
 &$=$&$1_I$ & K3 \\
 \end{tabular}
\end{center}
}
\end{proof}

\subsection{Heterogeneous algebras}\label{Heterogeneous presentation}

\begin{definition}
\label{def: heterogeneous algebras}
A {\em heterogeneous tqBa} (htqBa)  is a tuple $\mathbb{H} = (\mathbb{D}, \mathbb{L_I},\mathbb{L_C}, e_I, e_C, \iota, \gamma)$ such that:
\begin{itemize}
\item[H1] $\mathbb{D} = (\mathrm {D}, \lor, \land, \neg, \top, \bot)$ is a De Morgan algebra;
\item[H2] $\mathbb{L_I} = (\mathrm {L_I}, \cup, \cap, 0_I, 1_I)$ and $\mathbb{L_C} =(\mathrm{L_C}, \sqcup, \sqcap, 0_C, 1_C)$  are  bounded distributive  lattices;
\item[H3] $e_I: \mathbb{L_I}\hookrightarrow \mathbb{D}$ and $e_C: \mathbb{L_C}\hookrightarrow \mathbb{D}$  are  lattice homomorphisms;
\item[H4]  $\iota: \mathbb{D}\rightarrow \mathbb{L_I}$ and $ \gamma: \mathbb{D}\rightarrow\mathbb{L_C}$  satisfy the following identities:
\begin{center}
\begin{tabular}{llll}
(1)&$\iota(a \land b) = \iota(a) \cap \iota(b)$ &$\iota(\top) = 1$&$ \iota(\bot) = 0$\\
(2)&$\gamma(a \lor  b) = \gamma(a) \sqcup \gamma(b) $&$\gamma(\top) = 1$&$\gamma(\bot) = 0$
\end{tabular}
\end{center}
\item[H5] $e_I \dashv \iota$ \quad\quad $\gamma \dashv e_C$ \quad \quad $\iota(e_I(\alpha)) = \alpha$ \quad\quad$\gamma(e_C(\xi)) =\xi$;\footnote{Condition  H5 implies that $\iota$ is surjective and $e$ is injective.}
\item[H6] $e_C\gamma (a)= \neg e_I\iota (\neg a)$.
\end{itemize}
{\fns
\begin{center}
\begin{tikzpicture}[node/.style={circle, draw, fill=black}, scale=1]
\node (A) at (-1.5,-1.25) {$\mathbb{L_I}$};
\node (L) at (1.5,-1.25) {$\mathbb{D}$};
\node (B) at (4.5,-1.25) {$\mathbb{L_C}$};
\node (adj) at (0,-1.1) {{\rotatebox[origin=c]{270}{$\vdash$}}};
\node (adj) at (3,-1.1) {{\rotatebox[origin=c]{90}{$\vdash$}}};
\draw [right hook->] (A)  to node[below] {$e_{I}$}  (L);
\draw [left hook->] (B)  to node[below] {$e_{C}$}  (L);
\draw [->>] (L) to [out=135,in=45, looseness=1]   node[above] {$\iota$}  (A);
\draw [<<-] (B) to [out=135,in=45, looseness=1]   node[above] {$\gamma$}  (L);
\end{tikzpicture}
\end{center}
}
The heterogeneous algebras corresponding to the subclasses of tqBas considered in Section \ref{ssec: varieties} are defined as follows:
{\fns
\begin{center}
\begin{tabular}{|c|c|c|}
\hline
Algebra & Acronym & Conditions\\
\hline
{\em heterogeneous tqBa5} &htqBa5& H7: $ \mathbb{L_I} = \mathbb{L_C} = \mathbb{L}$ is a De Morgan algebra, $e_I = e_C = e$ is a De Morgan homomorphism.\\
\hline
{\em heterogeneous IA1} & hIA1 &H7, H8: $\mathbb{L}$ is a Boolean algebra.\\
 \hline
 {\em heterogeneous IA2} &hIA2 &H7, H9: $\iota(a \lor b) = \iota(a) \cup \iota(b)$.\\
 \hline
 {\em heterogeneous IA3}  &hIA3 &H7, H10: $\iota(a) \leq \iota(b)$ and $ \gamma(a) \leq  \gamma(b)$ imply $a \leq b$.\\
 \hline
 {\em heterogeneous pra} &hpra& H7, H8, H9, H10.\\
 \hline
 \end{tabular}
 \end{center}
 }
In what follows, we use the abbreviated names of the heterogeneous algebras written in ``blackboard bold''  (e.g.~$\mathbb{HTQBA}$, etc.) to indicate their corresponding classes. A heterogeneous algebra $\mathbb{H}$ is {\em perfect} if:
\begin{enumerate}
\item Every distributive  lattice (expansion) in the signature of $\mathbb{H}$ is  perfect (cf.~[Definition 2.14]\cite{GNV05});
\item Every join (resp.~meet) preserving map in the signature of $\mathbb{H}$ is completely join (resp.~meet) preserving.
\end{enumerate}
\end{definition}
\begin{definition}
\label{def: Dplus}
If  $\mathbb{T} = (\mathbb{L}, I)$ is a tqBa, we  let $\mathbb{T}^+ := (\mathbb{L}, \mathbb{K}_\mathbb{I},  \mathbb{K}_\mathbb{C}, e_I, e_C, \iota, \gamma)$, where:
\begin{itemize}
\item[$\cdot$]  $\mathbb{K}_\mathbb{I}$ and $\mathbb{K}_\mathbb{C}$ are the left and right kernels of $\mathbb{T}$ (cf.~Definition \ref{def:kernel});
\item[$\cdot$]  $e_I: \mathbb{K}_\mathbb{I} \hookrightarrow \mathbb{L}$ and $e_C: \mathbb{K}_\mathbb{C} \hookrightarrow \mathbb{L}$ are defined as the embeddings of the domains of $\mathbb{K}_\mathbb{I}$ and $\mathbb{K}_\mathbb{C}$ into the domain of $\mathbb{L}$;
\item[$\cdot$]  $\iota: \mathbb{L} \rightarrow \mathbb{K}_\mathbb{I}$ and $\gamma: \mathbb{L} \rightarrow \mathbb{K}_\mathbb{C}$ are defined by $\iota(a) =Ia$  and $\gamma(a) = Ca$ respectively.
\end{itemize}
If  $\mathbb{T} = (\mathbb{L}, I)$ is a tqBa5, the definition above can be simplified by identifying $\mathbb{K}_\mathbb{I}$ and $\mathbb{K}_\mathbb{C}$ and also $e_I$ and $e_C$. In this case  we write $\mathbb{T}^+ := (\mathbb{L}, \mathbb{K}, e, \iota, \gamma)$.
\end{definition}
\begin{definition}
\label{def:Htplus}
If $\mathbb{H} = (\mathbb{D}, \mathbb{L_I},\mathbb{L_C}, e_I, e_C, \iota, \gamma)$ is an htqBa, we let $\mathbb{H}_+ := (\mathbb{D}, I, C)$ where the unary operations $I$ and $C$ on $\mathbb{D}$ are defined by the assignments $a \mapsto e_I(\iota(a))$ and  $a \mapsto e_C(\gamma(a))$ respectively.

\end{definition}
\noindent Let $\mathbb{A}$ denote a class of rough algebras (cf.~Section \ref{ssec: varieties}), and $\mathbb{HA}$ its corresponding class of heterogeneous algebras.
\begin{proposition}
\label{prop:from single to multi}
\begin{enumerate}
\item If $\mathbb{T} \in \mathbb{A}$, then    $\mathbb{T}^+\in \mathbb{HA}$;
\item If $\mathbb{H}\in \mathbb{HA}$, then $\mathbb{H}_+\in \mathbb{A}$;
\item $ \mathbb{T} \cong (\mathbb{T}^+)_+ \quad \mbox{and}\quad \mathbb{H} \cong (\mathbb{H}_+)^+.$
\end{enumerate}
\end{proposition}
\subsection{Canonical extensions of heterogeneous algebras}
\label{ssec:canonical ext}
\label{Canonical extensions}
As discussed in other papers adopting the multi-type methodology, canonicity in the multi-type environment serves both to provide complete semantics for the analytic extensions of the basic logic (i.e.~extensions obtained by adding analytic inductive axioms) and to prove the conservativity  of their associated display calculi. In what follows, we let $\mathbb{D}^\delta$, $\mathbb{L}_\mathbb{I}^\delta$, and $\mathbb{L}_\mathbb{C}^\delta$ denote the canonical extensions of the algebras $\mathbb{D} $, $\mathbb{L}_\mathbb{I}$, and $\mathbb{L}_\mathbb{C}$ respectively, and $e_I^\delta$,  $e_C^\delta$, $\iota^\pi$, and $\gamma^\pi$  denote the extensions of $e_I$, $e_C$, $\iota$, and $\gamma$ respectively.\footnote{The order-theoretic properties of $e_I , e_C$, $\iota$ and $\gamma$ guarantee that they are {\em smooth}, that is, for each of them, $\sigma$-extension and $\pi$-extension coincide. However, the different notations in the superscripts are meant to emphasize that while the smoothness of the embeddings is used in the canonicity proofs, it is not needed in the case of  $\iota^\pi$ and $\gamma^\sigma$.}

\begin{definition}
\label{def: canonical heter}
If $\mathbb{H} =  (\mathbb{D}, \mathbb{L_I},\mathbb{L_C}, e_I, e_C, \iota, \gamma)\in \mathbb{HA}$ is an htqBa, then the {\em canonical extension} of $\mathbb{H}$ is the heterogeneous algebra $\mathbb{H^\delta} = (\mathbb{D}^\delta, \mathbb{K}_\mathbb{I}^\delta,  \mathbb{K}_\mathbb{C}^\delta, e_I^\delta , e_C^\delta, \iota^\pi, \gamma^\sigma)$.
\end{definition}
The following proposition is an immediate consequence of the fact that the defining conditions of the heterogeneous algebras of Definition \ref{def: heterogeneous algebras} can be expressed as analytic inductive inequalities (cf.~\cite[Definition 55]{greco2016unified}), and that each such inequality is canonical.
\begin{proposition}
\label{prop:canonical extensions}
If $\mathbb{H} \in \mathbb{HA}$, then $\mathbb{H^\delta}$ is a perfect element of $\mathbb{HA}$.
\end{proposition}

{\fns
\begin{center}
\begin{tikzpicture}[node/.style={circle, draw, fill=black}, scale=1]
\node (A) at (-1.5,-1.25) {$\mathbb{L_I}$};
\node (A delta) at (-1.5,1.5) {$\mathbb{L_I}^{\delta}$};
\node (L) at (1.5,-1.25) {$\mathbb{D}$};
\node (L delta) at (1.5,1.5) {$\mathbb{D}^{\delta}$};
\node (B) at (4.5,-1.25) {$\mathbb{L_C}$};
\node (B delta) at (4.5,1.5) {$\mathbb{L_C}^{\delta}$};
\node (adj) at (0,-0.9) {{\rotatebox[origin=c]{270}{$\vdash$}}};
\node (adju) at (0,1.9) {{\rotatebox[origin=c]{270}{$\vdash$}}};
\node (adju') at (0,0.83) {{\rotatebox[origin=c]{270}{$\vdash$}}};
\node (adju') at (3,0.83) {{\rotatebox[origin=c]{90}{$\vdash$}}};
\node (adj) at (3,-0.9) {{\rotatebox[origin=c]{90}{$\vdash$}}};
\node (adju) at (3,1.9) {{\rotatebox[origin=c]{90}{$\vdash$}}};

\draw [->>] (L delta)  to[out= 225,in= 315, looseness=1]  node[below] {$\iota'$}  (A delta);
\draw [->>] (L delta)  to[in= 225,out= 315, looseness=1]  node[below] {$\gamma'$}  (B delta);

\draw [right hook->] (A) to  (A delta);
\draw [right hook->] (L)  to (L delta);
\draw [right hook->] (B) to  (B delta);
\draw [right hook->] (A)  to node[below] {$e_I$}  (L);
\draw [right hook->] (A delta)  to node[below] {$e_{I}^{\delta}$}  (L delta);
\draw [left hook->] (B)  to node[below] {$e_C$}  (L);
\draw [left hook->] (B delta)  to node[below] {$e_{C}^{\delta}$}  (L delta);
\draw [->>] (L delta) to [out=135,in=45, looseness=1]   node[above] {$\iota^{\pi}$}  (A delta);
\draw [->>] (L) to [out=135,in=45, looseness=1]   node[above] {$\iota$}  (A);
\draw [<<-] (B delta) to [out=135,in=45, looseness=1]   node[above] {$\gamma^{\sigma}$}  (L delta);
\draw [<<-] (B) to [out=135,in=45, looseness=1]   node[above] {$\gamma$}  (L);
\end{tikzpicture}

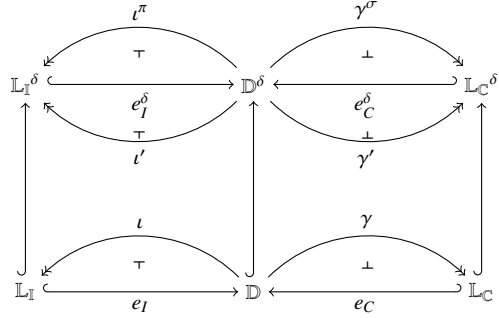
\textit {\textbf{Fig:}  Extending algebras to canonical extensions}
\end{center}
}
In Section \ref{ssec:soundness}, we prove that perfect elements of each class $\mathbb{HA}$ provide sound semantics for the multi-type calculus capturing the corresponding logic.

\section{Multi-type language for heterogeneous rough algebras}\label{sec: multi-type language}
 Heterogeneous algebras provide a natural interpretation for the following multi-type language $\mathcal{L}_{\mathrm{MT}}$ consisting of terms of types $\mathsf{D}$, $\mathsf{K_I}$ and $\mathsf{K_C}$.
$$\mathsf{D}\ni  A ::=  \,p \mid \, e_I(\alpha) \mid\, e_C(\xi) \mid\xtop \mid \xbot \mid A \xand A \mid A \xor A\mid\neg A $$
$$\mathsf{K_I}\ni \alpha ::= \, \iota(A)  \mid 1_I \mid 0_I \mid  \alpha \cup \alpha \mid  \alpha \cap \alpha  $$
$$\mathsf{K_C}\ni \xi ::= \, \gamma(A)  \mid 1_C\mid 0_C \mid \xi \sqcup \xi\mid  \xi \sqcap \xi.  $$
The logic $\mathrm{H.TQBA5}$ can be captured in a multi-type language consisting of the two  types $\mathsf{D}$ as above and $\mathsf{K}$ as follows:
$$\mathsf{K}\ni \alpha ::= \, \iota(A)\mid \gamma(A)  \mid 1 \mid 0 \mid{\sim}\alpha\mid \alpha \cup\alpha \mid  \alpha \cap \alpha . $$
The toggle between the single-type algebras and their corresponding heterogeneous algebras 
is reflected syntactically by the  translation $(\cdot)^t: \mathcal{L}\to \mathcal{L}_{\mathrm{MT}}$  defined as follows:
\begin{center}
\begin{tabular}{r c l c r c l r c l c r c l}
$p^t$ &$ = $& $p$&$\top^t$ &$ = $& $\xtop$ \\
$\bot^t$ &$ = $& $\xbot$&$(A\wedge B)^t$ &$ = $& $A^t \xand B^t$\\
$(A\vee B)^t$ &$ = $& $A^t \xor B^t$&$(\neg A)^t$ &$ = $& $\neg A^t$\\
$(IA)^t$ &$ = $& $e_I\iota (A^t)$&$(CA)^t$ &$ = $& $e_C\gamma( A^t)$&&\\
\end{tabular}
\end{center}
Recall that $\mathbb{T}^+$ denotes the heterogeneous algebra associated with the given algebra $\mathbb{T}$ (cf.~Definition \ref{def: Dplus}). The following proposition is proved by a routine induction on  $\mathcal{L}$-formulas.
\begin{proposition}
\label{prop:consequence preserved and reflected}
For all $\mathcal{L}$-formulas $A$ and $B$ and every  $\mathcal{L}$-algebra  $\mathbb{T}$,
$$\mathbb{T}\models A\leq B \quad \mbox{ iff }\quad \mathbb{T}^+ \models A^t\leq B^t.$$
\end{proposition}
We are now in a position to translate the axioms and rules of any logic $\mathrm{H}$ defined in Section \ref{ssec:logics} into $\mathcal{L}_{\mathrm{MT}}$.
\begin{center}
\begin{tabular}{rllc}
$Ia \fCenter a$ & $\rightsquigarrow$ &$e_I\iota a\leq a$ & (i)\\
$\top\fCenter I\top$& $\rightsquigarrow$ &$\top\leq e_I\iota(\top)$& (ii)\\
$I(a\land b)\fCenter Ia\land Ib$& $\rightsquigarrow$ &$e_I\iota(a \land b) \leq e_I\iota (a) \land e_I\iota (b)$& (iii)\\
$ Ia \land Ib \fCenter I(a\land b)$ & $\rightsquigarrow$ &$e_I\iota (a) \land e_I\iota(b) \leq e_I\iota(a \land b)$ & (iv)\\
$Ia\fCenter IIA$&$\rightsquigarrow$ &$e_I\iota(a) \leq e_I\iota e_I\iota(a)$ & (v)\\
$CIa\fCenter Ia$ &$\rightsquigarrow$ & $e_C\gamma (e_I\iota (a)) \leq e_I\iota (a)$& (vi)\\
$I(a\lor b) \fCenter Ia\lor Ib$&$\rightsquigarrow$ & $e_I\iota(a\lor b) \leq e_I\iota (a) \lor e_I\iota (b)$ & (vii)\\
 $Ia \lor Ib \fCenter I(a \lor b)$&$\rightsquigarrow$ & $e_I\iota(a) \lor e_I\iota(b) \leq e_I\iota(a \lor b)$& (viii)\\
 $\top \fCenter Ia\lor \neg Ia$&$\rightsquigarrow$ &$\top \leq e_I\iota(a) \lor \neg e_I\iota (a)$& (ix)\\
\AX $IA \fCenter IB$
\AX $CA \fCenter CB$
\BI  $A \fCenter B$
\DP &$\rightsquigarrow$ &$e_I\iota(a)\leq e_I\iota(b)$ and $e_C\gamma(a) \leq e_C\gamma(b)$ implies $a \leq b$&(x)
 \end{tabular}
 \end{center}
 Since $e_I$ and $e_C$ are order-embeddings, $e_I\iota(a)\leq e_I\iota(b)$ and $e_C\gamma(a) \leq e_C\gamma(b)$ are respectively equivalent to $\iota(a)\leq \iota(b)$ and $\gamma(a) \leq \gamma(b)$, and hence the quasi-inequality (x) can be equivalently rewritten as the following quasi-inequality, which defines the class $\mathbb{HIA}3$:
 \[\iota(a)\leq \iota(b)\mbox{ and }\gamma(a) \leq \gamma(b)\mbox{ implies }a \leq b.\]
 By applying adjunction, the inequalities in the antecedent can be equivalently rewritten as $a = e_C(\gamma(b))\wedge a$ and $b = b\vee e_I(\iota (a))$. Hence, the initial quasi-inequality can be equivalently rewritten as the following $\mathcal{L}_{\mathrm{MT}}$-inequality:
 \begin{equation}\label{eq:IA3rule}
 a \land e_C\gamma(b) \leq e_I\iota (a) \lor b.
 \end{equation}
The inequality above is analytic inductive, and hence it can be used, together with the other axioms of heterogeneous algebras, which, as observed in Section \ref{ssec:canonical ext}, are  analytic inductive, to generate the analytic structural rules of the calculi introduced in Section \ref{ssec:Display calculus}, with a methodology analogous to the one introduced in \cite{greco2016unified}. As we will discuss in Section \ref{ssec: completeness}, the inequalities (i)-(ix) are derivable in the appropriate calculi obtained in this way.
\section{Proper display calculi for the logics of rough algebras}
\label{ssec:Display calculus}
In the present section, we introduce  proper multi-type display calculi $\mathrm{D.A}$ for the logics associated with each class of algebras $\mathbb{A}$ mentioned in Section \ref{ssec: varieties}.  The language of these calculi has types $\mathsf{D}$ and $\mathsf{K_I}$ and $\mathsf{K_C}$, and  is built up from structural and operational (aka logical) connectives. Heterogeneous connectives $\wdia, \wbox, \bboxr, \bdial$ are interpreted as  $e_C, e_I, \iota, \gamma$ in heterogeneous algebras respectively. Each structural connective is denoted by decorating its corresponding logical connective  with $\hat{\phantom{a}}$ (resp.~$\check{\phantom{a}}$ or $\tilde{\phantom{a}}$).  
In what follows, we will adopt the convention that unary connectives bind more strongly than binary ones.
\subsection{Language}
\label{ssec:language of DtqBa}
\begin{itemize}
\item Structural and operational terms:
\end{itemize}
\begin{center}
\begin{tabular}{l}
$\mathsf{\phantom{_\wn}D} \left\{\begin{array}{l}
A  ::= \,p \mid \xtop \mid \xbot \mid \wbox \alpha \mid \wdia\xi \mid  \xneg A \mid A \land A \mid A \lor A \\
 \\
X ::= \,A \mid  \XBOT \mid \XTOP \mid \WBOX\Pi \mid \WDIA\Gamma  \mid  \XNEG X \mid  X \XAND X  \mid X \XOR X \mid X \XCRARR X \mid X \XRARR X
\end{array} \right.$
 \\
 \\
$\mathsf{K_I} \left\{\begin{array}{l}
\alpha ::= \,\bboxr A \mid \Itop  \mid \Ibot \mid \alpha \Iand \alpha \mid \alpha \Ior \alpha \mid (\Ineg\alpha)\\
\\
\Gamma ::= \alpha \mid \BBOXR X \mid \BBOXL X \mid \IBOT   \mid  \ITOP  \mid \Gamma \IAND \Gamma  \mid \Gamma \IOR \Gamma \mid \Gamma \ICRARR \Gamma  \mid \Gamma \IRARR \Gamma \mid (\tilde\INEG\Gamma) \\
\end{array} \right.$
 \\
 \\
$\mathsf{K_C} \left\{\begin{array}{l}
\xi ::= \,\bdial A \mid  \Ctop  \mid \Cbot \mid  \xi \Cand \xi \mid \xi \Cor \xi \mid (-\xi) \\
\\
\Pi ::=\, \xi \mid \BDIAL X \mid \BDIAR X \mid  \CBOT \mid \CTOP  \mid   \Pi \CAND \Pi  \mid \Pi \COR \Pi  \mid \Pi \CRARR \Pi \mid \Pi \CCRARR \Pi \mid (\tilde{-} \xi) \\
\end{array} \right.$
 \\
\end{tabular}
\end{center}

The formulas and structures in brackets  in the table above pertain to the language of $\mathrm{D.TQBA5}$ and its extensions.
\begin{itemize}
\item Interpretation of structural connectives as their logical counterparts\footnote{ In the synoptic table, the operational symbols which occur only at the structural level will appear between round brackets.}
\end{itemize}

\begin{enumerate}
\item structural and operational pure $\mathsf{D}$-type connectives:
{\fns
\begin{center}
\begin{tabular}{|r|c|c|c|c|c|c|c|c|}
\hline
structural operations&$\XTOP$ &$\XBOT$  & $\XAND$& $\XOR$ & $ \XNEG$ &$\XCRARR$ & $\XRARR$ \\
\hline
logical operations &$\xtop$ & $\xbot$ & $\xand$ & $\xor$ & $\xneg$ & $(\xcrarr)$& $(\xrarr)$ \\
\hline
\end{tabular}
\end{center}
}
\item structural and operational pure $\mathsf{K_I}$-type and $\mathsf{K_C}$-type connectives:
{\fns
\begin{center}
\begin{tabular}{|r|c|c|c|c|c|c||c|c|c|c|c|c|}
\hline
structural operations&$\ITOP$ &$\IBOT$  & $\IAND$& $\IOR$ &$\ICRARR$ & $\IRARR$ &$\CTOP$&$\CBOT$&$\hat{\sqcup}$&$\check{\sqcap}$&$\CRARR$&$\CCRARR$\\
\hline
logical operations &$\Itop$ & $\Ibot$ & $\Iand$ & $\Ior$ & $(\Icrarr)$& $(\Irarr)$&$\Ctop$&$\Cbot$&$\Cor$&$\Cand$&($\Crarr$)&($\Ccrarr$)\\
\hline
\end{tabular}
\end{center}
}
\item  As mentioned above, the language of $\mathrm{D.TQBA5}$ and its extensions includes the following structural and operational pure $\mathsf{K_I}$-type and $\mathsf{K_C}$-type connectives:
{\fns
\begin{center}
\begin{tabular}{|r|c|c|}
\hline
structural operations& $\tilde\INEG$  &$\tilde{-}$\\
\hline
logical operations & $\Ineg$&$-$\\
\hline
\end{tabular}
\end{center}
}
\item structural and operational multi-type connectives, and their algebraic counterparts:
{\fns
\begin{center}
\begin{tabular}{|r|c|c|c|c|c|c|c|c|c|}
\hline
types& \mc{2}{c|}{$\mathsf{D} \rightarrow \mathsf{K_I}$} &\mc{2}{c|}{$\mathsf{D} \rightarrow \mathsf{K_C}$} &$\mathsf{K_I} \rightarrow \mathsf{D}$&$\mathsf{K_C} \rightarrow \mathsf{D}$ \\
\hline
structural operations &$\BBOXL$& $\BBOXR$ &$\BDIAL$ &\BDIAR &$\WDIA$& $\WBOX$\\
\hline
logical operations &$(\bboxl)$&$\bboxr$ & $\bdial$ &$(\bdiar)$& $\wdia$ & $\wbox$\\
\hline
algebraic counterparts &$\iota'$&$\iota^\pi$ & $\gamma^\sigma$ &$\gamma'$& $e^\delta_I$ & $e^\delta_C$\\
\hline
\end{tabular}
\end{center}
}
\end{enumerate}

\subsection{Rules}
\label{ssec:rules}

In what follows, we will use $X, Y, W, Z$ as structural $\mathsf{D}$-variables, $\Gamma, \Delta, \Lambda$ as structural $\mathsf{K}_I$-variables, and $\Pi, \Sigma,\Omega$ as structural $\mathsf{K}_C$-variables.
The proper multi-type display calculus $\mathrm{D.TQBA}$ includes the following axiom and rules:
\begin{itemize}
\item Identity and Cut:
{\fns
\begin{center}
\begin{tabular}{llll}
\AXC{$$}
\LL{\fns $Id_\mathsf{D}$}
\UI $p \fCenter p$
\DP
&
\AX $X \fCenter A$
\AX $A \fCenter Y$
\RL{\fns $Cut_\mathsf{D}$}
\BI $X \fCenter Y$
\DP
&
\AX $\Gamma \fCenter \alpha$
\AX $\alpha \fCenter \Delta$
\LL{\fns $Cut_\mathsf{K_I}$}
\BI $\Gamma \fCenter \Delta$
\DP
&
\AX $\Pi \fCenter \xi$
\AX $\xi \fCenter \Sigma$
\RL{\fns $Cut_\mathsf{K_C}$}
\BI $\Pi \fCenter \Sigma$
\DP
\end{tabular}
\end{center}
}
\item Pure $\mathsf{D}$-type display rules:
{\fns
\begin{center}
\begin{tabular}{llll}

\AX $X \XAND Y \fCenter Z$
\LeftLabel{\scriptsize $res_\mathsf{D}$}
\doubleLine
\UI $Y \fCenter X \XRARR Z$
\DisplayProof
&
\AX $X \fCenter Y \XOR Z $
\RightLabel{\scriptsize $res_\mathsf{D}$}
\doubleLine
\UI$Y \XCRARR X \fCenter Z$
\DisplayProof
&
\AX $\XNEG X \fCenter Y$
\LeftLabel{\scriptsize $gal_\mathsf{D}$}
\doubleLine
\UI$\XNEG Y \fCenter X$
\DisplayProof
&
\AX $X \fCenter \XNEG Y$
\RightLabel{\scriptsize $gal_\mathsf{D}$}
\doubleLine
\UI  $ Y \fCenter \XNEG X$
\DisplayProof
\\
\end{tabular}
\end{center}
}
\item Pure $\mathsf{K_I}$-type and $\mathsf{K_C}$-type display rules:
{\fns
\begin{center}
\begin{tabular}{lclclcl}

\AX $\Gamma \IAND \Delta \fCenter \Lambda$
\LeftLabel{\scriptsize $res_\mathsf{K_I}$}
\doubleLine
\UI $\Delta \fCenter \Gamma \IRARR \Lambda$
\DisplayProof
&
&
\AX $\Gamma \fCenter \Delta \IOR \Lambda $
\RightLabel{\scriptsize $res_\mathsf{K_I}$}
\doubleLine
\UI$\Delta \ICRARR \Gamma \fCenter \Lambda$
\DisplayProof
&
&
\AX $\Pi \CAND \Sigma \fCenter \Omega$
\LeftLabel{\scriptsize $res_\mathsf{K_C}$}
\doubleLine
\UI $\Sigma \fCenter \Pi \CRARR \Omega$
\DisplayProof
&
&
\AX $\Pi \fCenter \Sigma \COR \Omega $
\RightLabel{\scriptsize $res_\mathsf{K_C}$}
\doubleLine
\UI$\Sigma \CCRARR \Pi \fCenter \Omega$
\DisplayProof
\end{tabular}
\end{center}
}
\item Multi-type display rules:
{\fns
\begin{center}
\begin{tabular}{llll}

\AX $\WDIA\Gamma  \fCenter Y$
\LeftLabel{\scriptsize $ad_\mathsf{DK_I}$}
\doubleLine
\UI $\Gamma \fCenter \BBOXR Y$
\DisplayProof
&
\AX $Y \fCenter \WDIA \Gamma$
\RightLabel{\scriptsize $ad_\mathsf{DK_I}$}
\doubleLine
\UI $\BBOXL Y \fCenter \Gamma$
\DisplayProof
&
\AX $Y \fCenter \WBOX \Pi$
\LeftLabel{\scriptsize $ad_\mathsf{DK_C}$}
\doubleLine
\UI$\BDIAL Y \fCenter \Pi$
\DisplayProof
&
\AX $\WBOX X \fCenter \Pi$
\RightLabel{\scriptsize $ad_\mathsf{DK_C}$}
\doubleLine
\UI$X \fCenter \BDIAR \Pi$
\DisplayProof
\end{tabular}
\end{center}
}
\item  Pure-type structural rules: these include standard Weakening (W), Contraction (C), Commutativity (E) and Associativity (A) in each type. We do not report on them.\footnote{In what follows, we use subscripts (indicating the type) to distinguish the rules for lattice operators in different type rules.}
{\fns
\begin{center}
\AX $X \fCenter  Y $
\RightLabel{\scriptsize $cont$}
\doubleLine
\UI $\XNEG Y \fCenter  \XNEG X$
\DP
\end{center}

\begin{center}
\begin{tabular}{llllll}
\AX $X  \fCenter Y$
\LeftLabel{\scriptsize $\XTOP$}
\doubleLine
\UI $X \XAND \XTOP   \fCenter Y$
\DisplayProof
&
\AX $X \fCenter  Y $
\RightLabel{\scriptsize $\XBOT$}
\doubleLine
\UI $X \fCenter  Y \XOR \XBOT$
\DisplayProof
&
\AX $\Gamma \fCenter \Delta$
\LeftLabel{\scriptsize $\ITOP$}
\doubleLine
\UI $\Gamma \IAND \ITOP   \fCenter \Delta$
\DisplayProof
&
\AX $\Gamma \fCenter  \Delta $
\RightLabel{\scriptsize $\IBOT$}
\doubleLine
\UI $\Gamma \fCenter  \Delta \IOR \IBOT$
\DisplayProof
&
\AX $\Pi  \fCenter \Sigma$
\LeftLabel{\scriptsize $\CTOP$}
\doubleLine
\UI $\Pi \CAND \CTOP   \fCenter \Sigma$
\DisplayProof
&

\AX $\Pi \fCenter  \Sigma $
\RightLabel{\scriptsize $\CBOT$}
\doubleLine
\UI $\Pi \fCenter \Sigma \COR \CBOT$
\DisplayProof

\end{tabular}
\end{center}
}
\item Multi-type structural rules:
{\fns
\begin{center}
\begin{tabular}{llll}
\AX $\XTOP \fCenter Y$
\LeftLabel{\scriptsize $\WDIA\ITOP$}
\doubleLine
\UI $\WDIA\ITOP \fCenter Y$
\DP
&
\AX $\Gamma  \fCenter \BBOXR \XBOT$
\RL{\scriptsize $\BBOXR\XTOP$}
\doubleLine
\UI $\Gamma  \fCenter \IBOT$
\DisplayProof
&
\AX $\BDIAL\XTOP \fCenter \Pi$
\LeftLabel{\scriptsize $\BDIAL\XTOP$}
\doubleLine
\UI $\CTOP \fCenter \Pi$
\DP
&
\AX $X  \fCenter  \XBOT$
\RL{\scriptsize $\WBOX\CBOT$}
\doubleLine
\UI $X  \fCenter \WBOX\CBOT$
\DisplayProof
\\
\\
\AX $\WDIA \Gamma \fCenter \WDIA \Delta$
\LeftLabel{\scriptsize $\WDIA$}
\doubleLine
\UI $\Gamma \fCenter \Delta$
\DP
&
\AX $\WBOX\Pi \fCenter \WBOX\Sigma$
\RL{\scriptsize $\WBOX$}
\doubleLine
\UI $\Pi \fCenter \Sigma$
\DP
&
\AX$\BDIAL X \fCenter \BDIAR Y$
\LL{\fns $\BDIAL\BDIAR$}
\UI$X \fCenter Y$
\DP
&
\AX$\BBOXL X \fCenter \BBOXR Y$
\RL{\fns $\BBOXL\BBOXR$}
\UI$X \fCenter Y$
\DP

\\
\\
\AX $X \fCenter \WDIA\BBOXR\XNEG Y$
\LL{\fns $IC$}
\doubleLine
\UI  $X \fCenter \XNEG\WBOX\BDIAL Y$
\DP
&
\AX $X \fCenter \WBOX\BDIAR\XNEG Y$
\RL{\fns $IC$}
\doubleLine
\UI  $X \fCenter \XNEG\WDIA\BBOXL Y$
\DP
&
\AX $\WBOX\BDIAL\XNEG X \fCenter  Y$
\LL{\fns $CI$}
\doubleLine
\UI  $\XNEG\WDIA\BBOXR X \fCenter  Y$
\DP
&
\AX $\WDIA\BBOXL\XNEG X \fCenter  Y$
\RL{\fns $CI$}
\doubleLine
\UI  $\XNEG\WBOX\BDIAR X \fCenter  Y$
\DP
\\
\\


\end{tabular}
\end{center}
}
\item Operational rules: those for the pure-type connectives are standard and omitted; those for multi-type connectives:
{\fns
\begin{center}
\begin{tabular}{lllll}
\AX $\BBOXL A\fCenter \Gamma$
\LL{\fns $\bboxl$}
\UI$\bboxl A \fCenter \Gamma$
\DP
&
\AX $\Gamma \fCenter A$
\RL{\fns $\bboxl$}
\UI$\BBOXL \Gamma \fCenter \bboxl A$
\DP
&
\AX $\Gamma \fCenter \BBOXR A$
\RL{\fns $\bboxr$}
\UI$\Gamma \fCenter \bboxr A$
\DP
&
\AX $A \fCenter \Gamma$
\LL{\fns $\bboxr$}
\UI$\bboxr A \fCenter \BBOXR \Gamma$
\DP
&
\\
\\
\AX $\BDIAL  A\fCenter \Pi$
\LL{\fns $\bdial$}
\UI$ \bdial A \fCenter \Pi$
\DP
&
\AX $\Pi \fCenter A$
\RL{\fns $\bdial$}
\UI$\BDIAL\Pi \fCenter \bdial A$
\DP
&
\AX $\Pi \fCenter \BDIAR A$
\RL{\fns $\bdiar$}
\UI$\Pi \fCenter \bdiar A$
\DP
&
\AX $A \fCenter \Pi$
\LL{\fns $\bdiar$}
\UI$\bdiar A \fCenter \BDIAR \Pi$
\DP
\\
\\
\AX $X \fCenter \WDIA \alpha$
\RL{\fns $\wdia$}
\UI$X \fCenter \wdia \alpha$
\DP
&
\AX$ \WDIA \alpha\fCenter X$
\LL{\fns $\wdia$}
\UI$ \wdia \alpha\fCenter X$
\DP
&
\AX $X \fCenter \WBOX \xi$
\RL{\fns $\wbox$}
\UI$X \fCenter \wbox \xi$
\DP
&
\AX$ \WBOX \xi \fCenter X$
\LL{\fns $\wbox$}
\UI$ \wbox \xi \fCenter X$
\DP
&
\end{tabular}
\end{center}
}
\end{itemize}

The calculus $\mathrm{D.TQBA5}$ is obtained by adding the following  rules to $\mathrm{D.TQBA}$:

\begin{itemize}
\item Display rules:
{\fns
\begin{center}
\begin{tabular}{llll}
\AX $\tilde\INEG \Gamma \fCenter \Delta$
\LeftLabel{\scriptsize $gal_\mathsf{K_I}$}
\doubleLine
\UI$\tilde\INEG \Delta \fCenter \Gamma$
\DisplayProof
&
\AX $\Gamma \fCenter \tilde\INEG \Delta$
\RightLabel{\scriptsize $gal_\mathsf{K_I}$}
\doubleLine
\UI  $ \Delta \fCenter \tilde\INEG \Gamma$
\DisplayProof
&
\AX $\CNEG \Pi \fCenter \Sigma$
\LeftLabel{\scriptsize $gal_\mathsf{K_C}$}
\doubleLine
\UI$\CNEG \Sigma \fCenter \Pi$
\DisplayProof
&
\AX $\Pi\fCenter \CNEG \Sigma$
\RightLabel{\scriptsize $gal_\mathsf{K_C}$}
\doubleLine
\UI  $ \Sigma \fCenter \CNEG \Pi$
\DisplayProof
\end{tabular}
\end{center}
}
\item Pure $\mathsf{K_I}$-type  and $\mathsf{K_C}$-type structural rules:
{\fns
\begin{center}
\begin{tabular}{ll}
\AX $\Gamma \fCenter  \Delta$
\RightLabel{\scriptsize $cont_I$}
\doubleLine
\UI $\tilde\INEG \Delta \fCenter  \tilde \INEG \Gamma$
\DP
&
\AX $\Pi \fCenter  \Sigma$
\RightLabel{\scriptsize $cont_C$}
\doubleLine
\UI $\CNEG \Sigma \fCenter  \CNEG \Pi$
\DP
\end{tabular}
\end{center}
}
\item Multi-type structural rules:
{\fns
\begin{center}
\begin{tabular}{llll}
\AX $X \fCenter \WDIA\BBOXR Y$
\RL{\scriptsize $\WDIA\BBOXR$}
\doubleLine
\UI $X \fCenter \WBOX\BDIAR Y$
\DisplayProof
&
\AX $X \fCenter  \WDIA\tilde\INEG \Gamma $
\RightLabel{\scriptsize ${\WDIA\tilde\INEG}$}
\doubleLine
\UI $X \fCenter  \XNEG\WDIA \Gamma$
\DP
&
\AX $X \fCenter  \WBOX\CNEG \Pi$
\RightLabel{\scriptsize ${\WBOX\CNEG}$}
\doubleLine
\UI $X \fCenter  \XNEG\WBOX \Pi$
\DP
\\
\end{tabular}
\end{center}

}

\item Additional operational rules for $\Ineg$ and $\Cneg$:
{\fns
\begin{center}
\begin{tabular}{llll}
\AX $\Gamma \fCenter \tilde\INEG\alpha $
\UI $\Gamma \fCenter \Ineg\alpha$
\DisplayProof
&
\AX $\tilde\INEG\alpha \fCenter \Gamma$
\UI $\Ineg\alpha\fCenter \Gamma$
\DisplayProof
&
\AX $\Pi \fCenter \CNEG\xi $
\UI $\Pi \fCenter \Cneg\xi$
\DisplayProof
&
\AX $\CNEG\alpha \fCenter \xi$
\UI $\Cneg\alpha\fCenter \xi$
\DisplayProof
\\
\end{tabular}
\end{center}

}

The proper display calculi for the axiomatic extensions of $\mathrm{H.TQBA5}$ discussed in \ref{ssec:logics} is obtained in following way.
\end{itemize}
{\fns
\begin{center}
\begin{tabular}{|c|c|p{7.2cm}|}
\hline
Name of logic&Display Calculus &\mc{1}{c|}{Rules}\\
\hline
$\mathrm{H.IA1}$& $\mathrm{D.IA1}$& \rule{0pt}{5ex}\noindent \begin{tabular}{ll}
 \AX $\Gamma \IAND \Delta \fCenter \Lambda$
\LeftLabel{\scriptsize $cgri$}
\UI $\Delta \fCenter \tilde\INEG\Gamma \IOR  \Lambda$
\DisplayProof
&
 \AX $\Pi \CAND \Sigma \fCenter \Omega$
\RL{\scriptsize $cgri$}
\UI $\Sigma \fCenter \CNEG\Pi \COR  \Omega$
\DisplayProof
\end{tabular}

\\
 \hline
$\mathrm{H.IA2}$&$\mathrm{D.IA2}$&\rule{0pt}{5ex}\noindent  \begin{tabular}{ll}
 \AX $\Gamma \fCenter \BBOXR (X \XOR Y)$
\RL{\fns $\BBOXR\IOR$}
\UI $\Gamma \fCenter\BBOXR X \IOR \BBOXR Y$
\DP
&
 \AX $\BDIAL (X \XAND Y) \fCenter \Pi$
\LL{\fns $\BDIAL\XAND$}
\UI $\BDIAL X \CAND \BDIAL Y \fCenter \Pi$
\DP
\end{tabular}

\\
 \hline
$\mathrm{H.IA3}$&$\mathrm{D.IA3}$& \rule{0pt}{5ex}\noindent
\AX $X \fCenter Y$
\AX $W \fCenter Z$
\RightLabel{\scriptsize $ia3$}
\BI $X \XAND\WBOX\BDIAL W \fCenter \WDIA\BBOXR Y\XOR Z$
\DP

\\
 \hline
 $\mathrm{H.PRA}$ &$\mathrm{D.PRA}$ &$cgri$, $\BDIAL\XAND$, $\BBOXR\XOR$, $pra$.\\
 \hline
 \end{tabular}
 \end{center}
}
\section{Properties}
Throughout this section, we let $\mathrm{H}$ denote any of the logics defined in Section \ref{ssec:logics}; let $\mathbb{A}$  and $\mathbb{HA}$ denote its corresponding class of single-type and heterogeneous algebras, respectively, and let $\mathrm{D.A}$ denote the  display calculus for $\mathrm{H}$.
\label{ssec: properties}
\subsection{Soundness for perfect $\mathbb{HA}$ algebras}
\label{ssec:soundness}
In the present subsection, we outline the  verification of the soundness of the rules of $\mathrm{D.A} $ w.r.t.~the semantics of {\em perfect} elements of $\mathbb{HA}$ (see Definition \ref{def: heterogeneous algebras}). The first step consists in interpreting structural symbols as logical symbols according to their (precedent or succedent) position, as indicated at the beginning of Section \ref{ssec:Display calculus}. This makes it possible to interpret sequents as inequalities, and rules as quasi-inequalities. For example, the rules on the left-hand side below are interpreted as the quasi-inequalities on the right-hand side:

\begin{center}
\begin{tabular}{lll}
\AX $X \fCenter Y$
\AX $W \fCenter Z$
\RightLabel{\scriptsize $pra$}
\BI $X \XAND\WBOX\BDIAL W \fCenter \WDIA\BBOXR Y\XOR Z$
\DP
&$\quad\rightsquigarrow\quad$&
$\forall a\forall b \forall c \forall d [(a \leq c \ \&\ b \leq d) \Rightarrow a \land e_C\gamma(b) \leq e_I\iota(c) \lor d].$
\end{tabular}
\end{center}
The verification of the soundness of the rules of $\mathrm{D.A}$ then consists in verifying the validity of their corresponding quasi-inequalities in any perfect element  of $\mathbb{HA}$. The verification of the soundness of pure-type rules and of the introduction rules following this procedure is routine, and is omitted.
The soundness of the rule $pra$ above is verified by the following ALBA-reduction, which shows that the quasi-inequality above is equivalent  to the inequality \eqref{eq:IA3rule}, which, as discussed in Section \ref{sec: multi-type language}, is valid on every $\mathbb{H}\in \mathbb{HIA}3$.
\begin{center}
\begin{tabular}{r l l}
   & $\forall p\forall q  [ p \land e_C\gamma(q) \leq e_I\iota (p) \lor q]$ &\\
iff & $\forall p\forall q\forall a\forall b \forall c \forall d   [ (a\leq p \ \& \ b\leq q\ \& \ p\leq c \ \&\ q\leq d)\Rightarrow a \land e_C\gamma(b) \leq e_I\iota (c) \lor d]$ &\\
iff & $\forall a\forall b \forall c \forall d   [ (a\leq c \ \& \ b\leq d)\Rightarrow a \land e_C\gamma(b) \leq e_I\iota (c) \lor d]$. &\\
\end{tabular}
\end{center}
The validity of the quasi-inequalities corresponding to the remaining  structural rules follows in an analogous way.

\subsection{Completeness}\label{ssec: completeness}
 Let  $A^\tau\vdash B^\tau$ be the translation of any $\mathcal{L}$-sequent $A\vdash B$  into the language of $\mathrm{D.A}$ which composes the  translation introduced in Section \ref{sec: multi-type language} with the correspondence between algebraic operations  and logical connectives indicated in table (iv) of Section \ref{ssec:language of DtqBa}.

\begin{proposition}
For every $\mathrm{H}$-derivable sequent $A \fCenter B$, the sequent $A^{\tau} \fCenter B^{\tau}$ is derivable in $\mathrm{D.A}$.
\end{proposition}
We only show the derivations of axioms T6, T7 and rule T8.
\begin{itemize}
\item[T6.] $\top \vdash IA \xor \neg IA  \quad \rightsquigarrow \quad \xtop \vdash \wdia\bboxr A \xor \xneg\wdia\bboxr A  $
{\fns
\[
\AX $A  \fCenter  A$
\LeftLabel{\scriptsize $\bboxr$}
\UI $ \bboxr A  \fCenter  \BBOXR A$
\RightLabel{\scriptsize $\bboxr$}
\UI $ \bboxr A  \fCenter  \bboxr A$
\RightLabel{\scriptsize $\ITOP$}
\UI $ \bboxr A \IAND \ITOP \fCenter  \bboxr A$
\RightLabel{\scriptsize cgri}
\UI $\ITOP \fCenter  \tilde{\INEG}\bboxr A \IOR \bboxr A$
\LL{\scriptsize $\WDIA$}
\UI $\WDIA\ITOP \fCenter  \WDIA(\tilde{\INEG}\bboxr A \IOR \bboxr A)$
\dashedLine
\UIC{\fns $\WDIA$ + $ad_\mathsf{DKI}$ + $W$ + $C$ + $\BBOXL\BBOXR$ }
\dashedLine
\UI $\WDIA\ITOP \fCenter  \WDIA\tilde{\INEG}\bboxr A \XOR  \WDIA\bboxr A $
\RightLabel{\scriptsize $res_\mathsf{D}$}
\UI $  \WDIA\tilde{\INEG}\bboxr A \XCRARR \WDIA\ITOP \fCenter  \WDIA\bboxr A $
\RightLabel{\scriptsize $\wdia$}
\UI $  \WDIA\tilde{\INEG}\bboxr A \XCRARR \WDIA\ITOP \fCenter  \wdia\bboxr A $
\RightLabel{\scriptsize $res_\mathsf{D}$}
\UI $ \WDIA\ITOP \fCenter  \WDIA\tilde{\INEG}\bboxr A \XOR  \wdia\bboxr A$
\RightLabel{\scriptsize $E_\mathsf{D}$}
\UI $ \WDIA\ITOP \fCenter  \wdia\bboxr A \XOR \WDIA\tilde{\INEG}\bboxr A$
\RightLabel{\scriptsize $ \WDIA\ITOP$}
\UI $ \XTOP \fCenter  \wdia\bboxr A \XOR \WDIA\tilde{\INEG}\bboxr A$
\LeftLabel{\scriptsize $ \xtop$}
\UI $ \xtop \fCenter  \wdia\bboxr A \XOR \WDIA\tilde{\INEG}\bboxr A$
\RightLabel{\scriptsize $res_\mathsf{D}$}
\UI $\wdia\bboxr A \XCRARR \xtop \fCenter \WDIA\tilde{\INEG}\bboxr A $
\RightLabel{\scriptsize $ \WDIA\tilde{\INEG}$}
\UI $\wdia\bboxr A \XCRARR \xtop \fCenter \XNEG\WDIA\bboxr A $
\RightLabel{\scriptsize $res_\mathsf{D}$}
\UI $\xtop \fCenter \wdia\bboxr A \XOR \xneg\wdia\bboxr A $
\RightLabel{\scriptsize $ \xor$}
\UI $\xtop \fCenter \wdia\bboxr A \xor \xneg\wdia\bboxr A $
\DP
\]
}
\item[T7.] $I(A \xor B) \dashv\vdash IA \xor IB \quad \rightsquigarrow \quad \wdia\bboxr (A \xor B) \dashv\vdash \wdia\bboxr A \xor  \wdia\bboxr B$
{\fns
\begin{center}
\begin{tabular}{cc}
\AX$A \fCenter A$
\AX$B \fCenter B$
\LeftLabel{\scriptsize $ \xor$}
\BI$A \xor B \fCenter A \XOR B$
\LeftLabel{\scriptsize $\bboxr$}
\UI$\bboxr (A \xor B) \fCenter \BBOXR (A \XOR B)$
\RL{\fns $\BBOXR\IOR$}
\UI$\bboxr (A \xor B) \fCenter \BBOXR A \IOR \BBOXR B$
\RightLabel{\scriptsize $res_\mathsf{K_I}$}
\UI$\BBOXR A \ICRARR \bboxr (A \xor B) \fCenter \BBOXR B$
\RL{\scriptsize $\bboxr$}
\UI$\BBOXR A \ICRARR \bboxr (A \xor B) \fCenter \bboxr B$
\RightLabel{\scriptsize $res_\mathsf{K_I}$}
\UI$\bboxr (A \xor B) \fCenter \BBOXR A \IOR \bboxr B$
\RightLabel{\scriptsize $E_\mathsf{K_I}$}
\UI$\bboxr (A \xor B) \fCenter \bboxr B \IOR \BBOXR A$
\RightLabel{\scriptsize $res_\mathsf{K_I}$}
\UI$\bboxr B \ICRARR \bboxr (A \xor B) \fCenter \BBOXR A$
\RL{\scriptsize $\bboxr$}
\UI$\bboxr B \ICRARR \bboxr (A \xor B) \fCenter \bboxr A$
\RightLabel{\scriptsize $res_\mathsf{K_I}$}
\UI$\bboxr (A \xor B) \fCenter \bboxr B \IOR \bboxr A$
\RightLabel{\scriptsize $E_\mathsf{K_I}$}
\UI$\bboxr (A \xor B) \fCenter \bboxr A \IOR \bboxr B$
\LL{\scriptsize $\WDIA$}
\UI$\WDIA \bboxr (A \xor B) \fCenter \WDIA (\bboxr A \IOR \bboxr B)$
\dashedLine
\UIC{\fns $\WDIA$ + $ad_\mathsf{DKI}$ + $W$ + $C$ + $\BBOXL\BBOXR$ }
\dashedLine
\UI$\WDIA \bboxr (A \xor B) \fCenter \WDIA \bboxr A \XOR \WDIA \bboxr B$
\dashedLine
\UIC{\fns $res_\mathsf{D}$+ $\WDIA$+$E$ }
\dashedLine
\UI$\wdia \bboxr (A \xor B) \fCenter \wdia \bboxr A \xor \wdia \bboxr B$
\DP
 &
\AX$A \fCenter A$
\RightLabel{\scriptsize $W_\mathsf{D}$}
\UI$A \fCenter A \XOR B$
\RL{\scriptsize $ \xor$}
\UI$A \fCenter A \xor B$
\LeftLabel{\scriptsize $\bboxr$}
\UI$\bboxr A \fCenter \BBOXR (A \xor B)$
\RL{\scriptsize $\bboxr$}
\UI$\bboxr A \fCenter \bboxr (A \xor B)$
\LL{\scriptsize $\WDIA$}
\UI$\WDIA \bboxr A \fCenter \WDIA \bboxr (A \xor B)$
\LL{\scriptsize $\wdia$}
\UI$\wdia \bboxr A \fCenter \WDIA \bboxr (A \xor B)$
\RL{\scriptsize $\wdia$}
\UI$\wdia \bboxr A \fCenter \wdia \bboxr (A \xor B)$
\AX$B \fCenter B$
\RightLabel{\scriptsize $W_\mathsf{D}$}
\UI$B \fCenter B \XOR A$
\RightLabel{\scriptsize $E_\mathsf{D}$}
\UI$B \fCenter A \XOR B$
\RL{\scriptsize $ \xor$}
\UI$B \fCenter A \xor B$
\LeftLabel{\scriptsize $\bboxr$}
\UI$\bboxr B \fCenter \BBOXR (A \xor B)$
\RL{\scriptsize $\bboxr$}
\UI$\bboxr B \fCenter \bboxr (A \xor B)$
\LL{\scriptsize $\WDIA$}
\UI$\WDIA \bboxr B \fCenter \WDIA \bboxr (A \xor B)$
\LL{\scriptsize $\wdia$}
\UI$\wdia \bboxr B \fCenter \WDIA \bboxr (A \xor B)$
\RL{\scriptsize $\wdia$}
\UI$\wdia \bboxr B \fCenter \wdia \bboxr (A \xor B)$
\LL{\scriptsize $ \xor$}
\BI$\wdia \bboxr A \xor \wdia \bboxr B \fCenter \wdia \bboxr (A \xor B) \XOR \wdia \bboxr (A \xor B)$
\RightLabel{\scriptsize $C_\mathsf{D}$}
\UI$\wdia \bboxr A \xor \wdia \bboxr B \fCenter \wdia \bboxr (A \xor B)$
\DP
 \\
\end{tabular}
\end{center}
}
\newpage
\item[T8.] $IA\fCenter IB$ and $CA \fCenter CB$ imply $A \fCenter B \quad \rightsquigarrow \quad \wdia\bboxr A \fCenter \wdia\bboxr B$ and $\wbox\bdial A \fCenter \wbox\bdial B$ imply $A \fCenter B$
{\fns
\begin{center}
\begin{tabular}{c}
\AX$A \fCenter A$
\AX$\wbox \bdial A \fCenter \wbox \bdial B$
\dashedLine
\UIC{\fns $Id_{\mathrm{D}}$ + $\bdial$ + $\wbox$ + $Cut_{\mathrm{D}}$}
\dashedLine
\UI$\WBOX \BDIAL A \fCenter \WBOX \bdial B$
\LL{\fns $\WBOX$}
\UI$\BDIAL A \fCenter \bdial B$
\LL{\fns $ad_\mathsf{DK_C}$}
\UI$A \fCenter \WBOX \bdial B$
\RL{\fns $\wbox$}
\UI$A \fCenter \wbox \bdial B$
\RL{\fns $\xand$}
\BI$A \XAND A \fCenter A \xand \wbox \bdial B$
\LL{\fns $C_\mathsf{D}$}
\UI$A \fCenter A \xand \wbox \bdial B$

\AX$A \fCenter A$
\AX$B \fCenter B$
\RL{\fns $ia3$}
\BI$A \XAND \WBOX \BBOXL B \fCenter \WDIA \BBOXR A \XOR B$
\dashedLine
\UIC{\fns $\xand_l$ + $\xor_r$ }
\dashedLine
\UI$A \xand \wbox \bboxl B \fCenter \wdia \bboxr A \xor B$

\AX$\wdia \bboxr A \fCenter \wdia \bboxr B$
\dashedLine
\UIC{\fns $Id_{\mathrm{D}} + \bboxr+\wdia + Cut_{\mathrm{D}}$}
\dashedLine
\UI$\WDIA \bboxr A \fCenter \WDIA \BBOXR B$
\LL{\fns $\WDIA$}
\UI$\bboxr A \fCenter \BBOXR B$
\LL{\fns $ad_\mathsf{DK_I}$}
\UI$\WDIA \bboxr A \fCenter B$
\LL{\fns $\wdia$}
\UI$\wdia \bboxr A \fCenter B$
\AX$B \fCenter B$
\LL{\fns $\xor$}
\BI$\wdia \bboxr A \xor B \fCenter B \XOR B$
\RL{\fns $C_\mathsf{D}$}
\UI$\wdia \bboxr A \xor B \fCenter B$
\RL{\fns $Cut_{\mathsf{D}}$}
\BI$A \xand \wbox \bboxl B \fCenter B$
\RL{\fns $Cut_{\mathsf{D}}$}

\BI$A \fCenter B$
\DP
\end{tabular}
\end{center}
}
\end{itemize}
\subsection{Conservativity}\label{ssec: conservativity}
To argue that $\mathrm{D.A}$ is conservative w.r.t.~$\mathrm{H}$ we follow the standard proof strategy discussed in \cite{greco2016unified,GKPLori}. Let  $\vdash_{\mathrm{H}}$ denote the  syntactic consequence relation corresponding to $\mathrm {H}$ and $\models_{\mathbb{HA}}$ denote the  semantic consequence relation arising from (perfect) heterogeneous algebras in $\mathbb{HA}$. We need to show that, for all $\mathcal{L}$-formulas $A$ and $B$,  if $A^\tau \vdash B^\tau$ is a $\mathrm{D.A}$-derivable sequent, then  $A \vdash B$ is derivable in $\mathrm{H}$. This claim can be proved using  the following facts: (a) The rules of $\mathrm{D.A}$ are sound w.r.t.~perfect members of $\mathbb{HA}$ (cf.~Section \ref{ssec:soundness});  (b) $\mathrm{H}$ is complete w.r.t.~the class of perfect algebras in $\mathbb{A}$ (cf.~Proposition \ref{completeness:H.SDM});  (c) A perfect element of $\mathbb{A}$ is equivalently presented as  a perfect member of $\mathbb{HA}$ so that the semantic consequence relations arising from each type of structures preserve and reflect the translation (cf.~Proposition \ref{prop:consequence preserved and reflected}). Let $A, B$ be $\mathcal{L}$-formulas. If  $A^\tau \vdash B^\tau$ is $\mathrm{D.A}$-derivable, then by (a),  $\models_{\mathbb{HA}} A^\tau \vdash B^\tau$. By (c), this implies that $\models_{\mathbb{A}}  A\vdash B$, where $\models_{\mathbb{A}}$ denotes the semantic consequence relation arising from the perfect members of class $\mathbb{A}$. By (b), this implies that $A\vdash B$ is derivable in $\mathrm{H}$, as required.

\subsection{Cut elimination and subformula property}
 In the present section, we briefly sketch the proof of cut elimination and subformula property for $\mathrm{D.A}$. As hinted to earlier on, proper display calculi have been designed so that the cut elimination and subformula property  can  be inferred from a meta-theorem, following the strategy introduced by Belnap for display calculi. The meta-theorem to which we will appeal for each $\mathrm{D.A}$ was proved in \cite{TrendsXIII}.

All conditions in \cite[Theorem 4.1]{TrendsXIII} except $\mathrm{C}'_8$ are readily satisfied by inspecting the rules. Condition $\mathrm{C}'_8$ requires to check that reduction steps are available for every application of the cut rule in which both cut-formulas are principal, which either remove the original cut altogether or replace it by one or more cuts on formulas of strictly lower complexity.  In what follows, we only show  $\mathrm{C}'_8$ for the unary connectives.

\paragraph*{Pure $\mathsf{D}$-type connectives:}
{\fns
\begin{center}
\begin{tabular}{ccc}
\!\!\!\!\!
\bottomAlignProof
\AXC{\ \ \ $\vdots$ \raisebox{1mm}{$\pi_1$}}
\noLine
\UI$X \fCenter \XNEG A$
\RL{\fns $\xneg$}
\UI$X \fCenter \xneg A$

\AXC{\ \ \ $\vdots$ \raisebox{1mm}{$\pi_2$}}
\noLine
\UI$ \XNEG A \fCenter Y$
\LL{\fns $\xneg$}
\UI$\xneg A \fCenter Y$
\RL{\fns $Cut_\mathsf{D}$}
\BI$X \fCenter Y$
\DisplayProof

 & $\rightsquigarrow$ &

\!\!\!\!\!
\bottomAlignProof
\AXC{\ \ \ $\vdots$ \raisebox{1mm}{$\pi_2$}}
\noLine
\UI$\XNEG A \fCenter Y$
\LL{\fns $gal_\mathsf{D}$}
\UI$\XNEG Y \fCenter A$
\AXC{\ \ \ $\vdots$ \raisebox{1mm}{$\pi_1$}}
\noLine
\UI$X \fCenter \XNEG A$
\RL{\fns $gal_\mathsf{D}$}
\UI$A \fCenter \XNEG X$
\RL{\fns $Cut_\mathsf{D}$}
\BI$ \XNEG Y \fCenter \XNEG X$
\RL{\fns $cont$}
\UI$X \fCenter Y$
\DisplayProof
 \\
\end{tabular}
\end{center}
}
\noindent The cases for $\Ineg \alpha$ and $\Cneg\xi$ of D.TQBA5 and its extensions  are standard and similar to the one above.

\paragraph*{Multi-type connectives:}
{\fns
\begin{center}
\begin{tabular}{ccc}
\bottomAlignProof
\AXC{\ \ \ $\vdots$ \raisebox{1mm}{$\pi_1$}}
\noLine
\UI$\Gamma \fCenter \BBOXR A$
\RL{\fns $\bboxr$}
\UI$\Gamma \fCenter \bboxr A$
\AXC{\ \ \ $\vdots$ \raisebox{1mm}{$\pi_2$}}
\noLine
\UI$A \fCenter Y$
\LL{\fns $\xneg$}
\UI$\bboxr A \fCenter \BBOXR Y$
\BI$\Gamma \fCenter \BBOXR Y$
\DisplayProof

 & $\rightsquigarrow$ &

\!\!\!\!\!\!\!
\bottomAlignProof
\AXC{\ \ \ $\vdots$ \raisebox{1mm}{$\pi_1$}}
\noLine
\UI$\Gamma \fCenter  \BBOXR A$
\UI$\WDIA \Gamma \fCenter A$
\AXC{\ \ \ $\vdots$ \raisebox{1mm}{$\pi_2$}}
\noLine
\UI$A \fCenter Y$
\BI$\WDIA \Gamma \fCenter Y$
\UI$ \Gamma \fCenter \BBOXR Y$
\DisplayProof
 \\
\end{tabular}
\end{center}

\begin{center}
\begin{tabular}{ccc}
\bottomAlignProof
\AXC{\ \ \ $\vdots$ \raisebox{1mm}{$\pi_1$}}
\noLine
\UI$X \fCenter \WDIA \alpha$
\UI$X \fCenter \wdia \alpha$
\AXC{\ \ \ $\vdots$ \raisebox{1mm}{$\pi_2$}}
\noLine
\UI$\WDIA \alpha  \fCenter Y$
\UI$\wdia \alpha \fCenter Y$
\RL{\fns $Cut_\mathsf{D}$}
\BI$X \fCenter Y$
\DisplayProof

 & $\rightsquigarrow$ &

\!\!\!\!\!\!\!
\bottomAlignProof
\AXC{\ \ \ $\vdots$ \raisebox{1mm}{$\pi_1$}}
\noLine
\UI$X \fCenter \WDIA \alpha$
\RL{\fns $ad_\mathsf{DK_I}$}
\UI$\BBOXL X \fCenter \alpha$
\AXC{\ \ \ $\vdots$ \raisebox{1mm}{$\pi_2$}}
\noLine
\UI$\WDIA \alpha  \fCenter Y$
\LL{\fns $ad_\mathsf{DK_I}$}
\UI$ \alpha \fCenter \BBOXR Y$

\RL{\fns $Cut_\mathsf{K_I}$}
\BI$\BBOXL X \fCenter \BBOXR Y$
\RL{\fns $\BBOXL\BBOXR$}
\UI$X \fCenter Y$
\DisplayProof
\end{tabular}
\end{center}
}
\noindent The cases for $\bdial A$ and $\wbox \xi$ are analogous.

\bibliography{BIB}

\begin{thebibliography}{10}

\bibitem{banerjee1996rough}
Mohua Banerjee and Mihir~Kumar Chakraborty.
\newblock Rough sets through algebraic logic.
\newblock {\em Fundamenta Informaticae}, 28(3, 4):211--221, 1996.

\bibitem{BGPTW}
Marta B\'{i}lkov\'{a}, Giuseppe Greco, Alessandra Palmigiano, Apostolos
  Tzimoulis, and Nachoem~M. Wijnberg.
\newblock The logic of resources and capabilities.
\newblock {\em Review of Symbolic Logic}, 2018. Doi: 10.1017/S175502031700034X.

\bibitem{birkhoff1970heterogeneous}
Garrett Birkhoff and John~D. Lipson.
\newblock Heterogeneous algebras.
\newblock {\em Journal of Combinatorial Theory}, 8(1):115--133, 1970.

\bibitem{comer1995perfect}
Sthephen~D. Comer.
\newblock Perfect extensions of regular double stone algebras.
\newblock {\em Algebra Universalis}, 34(1):96--109, 1995.

\bibitem{davey2002introduction}
Brian~A. Davey and Hilary~A. Priestley.
\newblock {\em Introduction to lattices and order}.
\newblock Cambridge university press, 2002.

\bibitem{PDL}
Sabine Frittella, Giuseppe Greco, Alexander Kurz, and Alessandra Palmigiano.
\newblock Multi-type display calculus for propositional dynamic logic.
\newblock {\em Journal of Logic and Computation}, 26 (6):2067--2104, 2016.

\bibitem{TrendsXIII}
Sabine Frittella, Giuseppe Greco, Alexander Kurz, Alessandra Palmigiano, and
  Vlasta Sikimi\'{c}.
\newblock Multi-type sequent calculi.
\newblock {\em Proceedings Trends in Logic XIII, A. Indrzejczak, J. Kaczmarek,
  M. Zawidski eds}, 13:81--93, 2014.

\bibitem{Multitype}
Sabine Frittella, Giuseppe Greco, Alexander Kurz, Alessandra Palmigiano, and
  Vlasta Sikimi\'{c}.
\newblock A multi-type display calculus for dynamic epistemic logic.
\newblock {\em Journal of Logic and Computation}, 26 (6):2017--2065, 2016.

\bibitem{Inquisitive}
Sabine Frittella, Giuseppe Greco, Alessandra Palmigiano, and Fan Yang.
\newblock A multi-type calculus for inquisitive logic.
\newblock In {\em International Workshop on Logic, Language, Information, and
  Computation}, pages 215--233. Springer, 2016.

\bibitem{ProofTheoreticDEL}
Sabine Frittella, Kurz~Alexander Greco, Giuseppe, Alessandra Palmigiano, and
  Vlasta Sikimi\'{c}.
\newblock A proof-theoretic semantic analysis of dynamic epistemic logic.
\newblock {\em Journal of Logic and Computation}, 26(6):1961--2015, 2016.

\bibitem{GNV05}
Mai Gehrke, Hideo Nagahashi, and Yde Venema.
\newblock A {S}ahlqvist theorem for distributive modal logic.
\newblock {\em Annals of Pure and Applied Logic}, 131:65--102, 2005.

\bibitem{GKPLori}
Giuseppe Greco, Alexander Kurz, and Alessandra Palmigiano.
\newblock Dynamic epistemic logic displayed.
\newblock In Huaxin Huang, Davide Grossi, and Olivier Roy, editors, {\em
  Proceedings of the 4th International Workshop on Logic, Rationality and
  Interaction (LORI-4)}, volume 8196 of {\em LNCS}, 2013.

\bibitem{greco2017multi}
Giuseppe Greco, Fei Liang, M~Andrew Moshier, and Alessandra Palmigiano.
\newblock Multi-type display calculus for semi {D}e {M}organ logic.
\newblock In {\em International Workshop on Logic, Language, Information, and
  Computation}, pages 199--215. Springer, 2017.

\bibitem{bilattice}
Giuseppe Greco, Fei Liang, Alessandra Palmigiano, and Umberto Rivieccio.
\newblock {B}ilattice logic properly displayed.
\newblock {\em Fuzzy Sets and Systems}, doi: 10.1016/j.fss.2018.05.007, 2018.

\bibitem{greco2016unified}
Giuseppe Greco, Minghui Ma, Alessandra Palmigiano, Apostolos Tzimoulis, and
  Zhiguang Zhao.
\newblock Unified correspondence as a proof-theoretic tool.
\newblock {\em Journal of Logic and Computation}, 2016. doi:
  10.1093/logcom/exw022.

\bibitem{GrecoPalmigianoLatticeLogic}
Giuseppe Greco and Alessandra Palmigiano.
\newblock Lattice logic properly displayed.
\newblock In {\em International Workshop on Logic, Language, Information, and
  Computation}, pages 153--169. Springer, 2017.

\bibitem{linearlogPdisplayed}
Giuseppe Greco and Alessandra Palmigiano.
\newblock Linear logic properly displayed.
\newblock {\em Studia Logica, arXiv preprint: 1611.04184}, forthcoming.

\bibitem{iwinski1987algebraic}
Tadeusz~B. Iwinski.
\newblock Algebraic approach to rough sets.
\newblock {\em Bull. Pol. Acd. Sci., Math.}, 35:673--683, 1987.

\bibitem{fei2018}
Fei Liang.
\newblock {\em Multi-type {A}lgebraic {P}roof {T}heory}, PhD Dissertation,
  Delft University of Technology, 2018.

\bibitem{MinChaZhe18}
Minghui Ma, Mihir~Kumar Chakraborty, and Zhe Lin.
\newblock Sequent cacluli for varieties of topological quasi-{B}oolean
  algebras.
\newblock {\em Proceedings of IJCRS 2018: International Joint Conference on
  Rough Sets}, 2018, forthcoming.

\bibitem{KriPalNac18}
Krishna Manoorkar, Alessandra Palmigiano, and Nachoem~M. Wijnberg.
\newblock Rought concepts.
\newblock In preparation.

\bibitem{pawlak1982rough}
Zdzis{\l}aw Pawlak.
\newblock Rough sets.
\newblock {\em International journal of computer \& information sciences},
  11(5):341--356, 1982.

\bibitem{saha2014algebraic}
Anirban Saha, Jayanta Sen, and Mihir~Kumar Chakraborty.
\newblock Algebraic structures in the vicinity of pre-rough algebra and their
  logics.
\newblock {\em Information Sciences}, 282:296--320, 2014.

\bibitem{saha2016algebraic}
Anirban Saha, Jayanta Sen, and Mihir~Kumar Chakraborty.
\newblock Algebraic structures in the vicinity of pre-rough algebra and their
  logics ii.
\newblock {\em Information Sciences}, 333:44--60, 2016.

\bibitem{apostolos2018}
Apostolos Tzimoulis.
\newblock {\em Algebraic and {P}roof-{T}heoretic {F}oundations of the {L}ogics
  for {S}ocial {B}ehaviour}, PhD Dissertation, Delft University of Technology,
  2018.

\end{thebibliography}
\bibliographystyle{plain}
\end{document}